\documentclass[12pt, reqno]{amsart}
\usepackage{amsmath, amsthm, amscd, amsfonts, amssymb, graphicx, color, mathrsfs}
\usepackage[bookmarksnumbered, colorlinks, plainpages]{hyperref}

\newtheorem{theorem}{Theorem}[section]
\newtheorem{lemma}[theorem]{Lemma}

\newtheorem{corollary}[theorem]{Corollary}
\theoremstyle{definition}
\newtheorem{definition}[theorem]{Definition}
\newtheorem{example}[theorem]{Example}

\theoremstyle{remark}
\newtheorem{remark}[theorem]{Remark}
\numberwithin{equation}{section}

\begin{document}
\setcounter{page}{1}

\title[ The nuclear trace for vector-valued operators]{The nuclear trace of periodic vector-valued pseudo-differential operators  with applications to index theory }

\author[D. Cardona]{Duv\'an Cardona}
	\address{
		Duv\'an Cardona:
		\endgraf
		Department of Mathematics  
		\endgraf
		Pontificia Universidad Javeriana
		\endgraf
		Bogot\'a
		\endgraf
		Colombia
		\endgraf
		{\it E-mail address} {\rm 
			duvanc306@gmail.com}
	}
	
	\author[V. Kumar]{Vishvesh Kumar}
	\address{
		Vishvesh Kumar:
		\endgraf
		Department of Mathematics  
		\endgraf
		Indian Institute of Technology Delhi.
		\endgraf
		New Delhi - 110016
		\endgraf
		India
		\endgraf
		{\it E-mail address} {\rm vishveshmishra@gmail.com}
	}

%\address{$^{2}$ Department of Pure Mathematics, Ferdowsi University of Mashhad, P. O. Box 1159, Mashhad 91775, Iran;
%\newline
%Tusi Mathematical Research Group (TMRG), Mashhad, Iran.}
%\email{\textcolor[rgb]{0.00,0.00,0.84}{second@afa.ac.ir}}



\begin{abstract}
In this paper we investigate the nuclear trace of vector-valued Fourier multipliers on the torus and its applications to the index theory of periodic pseudo-differential operators. First we characterise the nuclearity of pseudo-differential operators acting on Bochner integrable functions. In this regards, we consider the periodic and the discrete cases. We go on to address  the problem of finding sharp sufficient conditions for the nuclearity of vector-valued Fourier multipliers on the torus. We end our investigation with two index formulae. First, we express the index of a vector-valued Fourier multiplier in terms of its operator-valued symbol and then we use this formula for expressing the index of certain elliptic operators belonging to periodic H\"ormander classes.
\end{abstract} \maketitle

\tableofcontents
\allowdisplaybreaks
\section{Introduction}
\subsection{Outline of the paper}
This contribution is devoted to study the nuclearity of vector-valued pseudo-differential operators on the torus. Our main goal is to investigate the nuclear trace of these operators and its applications to the index  theory of periodic operators. 

Let $\mathbb{T}^n\simeq \mathbb{R}^n/\mathbb{Z}^n$ be the $n$-dimensional torus and let $H$ be a (real or complex) Hilbert space. Denote the  algebra of bounded linear operators on $H$ by $\mathcal{L}(H)$. The vector-valued quantization of a function              $
    \sigma: \mathbb{T}^n\times \mathbb{Z}^n\rightarrow \mathcal{L}(H)$ is the operator defined by  
\begin{eqnarray} \label{Pseudotorus}
Af(x)\equiv T_\sigma f(x):=\sum_{\xi\in\mathbb{Z}^n}e^{i2\pi x\cdot \xi}\sigma(x,\xi)\widehat{f}(\xi), \,\,x:=(x_1,x_2,\cdots ,x_n)\in \mathbb{T}^n,
\end{eqnarray} for all $f\in C^{\infty}(\mathbb{T}^n,H). $ Here, $\widehat{f}$ is the \emph{vector-valued Fourier transform} of $f$ defined in the sense of Bochner as
\begin{equation}
    \widehat{f}(\xi)=\int\limits_{\mathbb{T}^n}e^{-i2\pi x\cdot \xi}f(x)dx\in H,\,\,\xi:=(\xi_1,\xi_2,\cdots ,\xi_n)\in\mathbb{Z}^n. 
\end{equation} 
Here $x\cdot \xi=x_1\xi_1+\cdots + x_n\xi_n$ denotes  the usual inner product on $\mathbb{R}^n.$ If $\dim H\geq \aleph_0,$ the function $\sigma$ is said to be  the \emph{operator-valued symbol} associated to the operator $A.$ For $2\leq  \dim H<\infty,$ $\mathcal{L}(H)\simeq \mathbb{C}^{\dim H\times \dim H }$ and  $\sigma$ becomes a \emph{matrix-valued symbol} associated to  $A.$ If $H$ is an one-dimensional complex Hilbert space,  then $H\simeq \mathbb{C}$ and under the identification  $\mathcal{L}(H)\simeq \mathbb{C},$ the symbol   $\sigma$ is scalar-valued. We will consider two cases, namely, $\dim H=\aleph_0$ (i.e. $H$ is a separable Hilbert space) or $\dim H=1,$ (i.e., $H=\mathbb{C}$).

Of particular interest are vector-valued Fourier multipliers  which are pseudo-differential operators with symbols $\sigma$ depending only on the frequency variable $\xi\in \mathbb{Z}^n.$ The vector-valued Fourier multipliers are characterised by the relation,
\begin{equation}
    \widehat{Af}(\xi)=\sigma(\xi)\widehat{f}(\xi),\,\,f\in C^{\infty}(\mathbb{T}^n,H),\,\,\widehat{f}(\xi)\in H,\,\,\,\sigma(\xi)\in\mathcal{L}( H),\,\,\,\xi\in\mathbb{Z}^n.
\end{equation}

The Fourier analysis on $L^2(\mathbb{T}^n,H)$ allow us to investigate the spectral and analytic properties of periodic pseudo-differential operators acting on $L^p(\mathbb{T}^n,H)$.  By the spectral properties of pseudo-differential operator we understand to classify them in  ideals of the algebra of  bounded operators and by analytical properties we refer to the problem of studying their mapping properties in (vector-valued and scalar-valued) $L^p$-spaces, Sobolev spaces, Besov spaces and anothers function spaces of interest in the analysis of PDE's.

In this paper we want to investigate some of the spectral properties of periodic pseudo-differential operators. To be precise, we would like to classify those vector-valued periodic pseudo-differential operators belonging to the ideal of $s$-nuclear, $0<s \leq 1,$ operators $\mathfrak{N}_{s}(L^{p_1}(\mathbb{T}^n,H),L^{p_2}(\mathbb{T}^n,H)).$ As an application of our classification, we study the Fredholm index of vector-valued Fourier multipliers and we consequently apply these index formulae to elliptic periodic pseudo-differential operators. Our main results will be presented in Subsection \ref{MainResults} as well as the notion of nuclearity and its relation with index formulae.  

The  literature for the analytical and spectral  properties in the vector-valued context need to be considered separately (in view of the cases $\dim H=\aleph_0$ or $\dim H=1)$). Indeed,
\begin{itemize}
    \item if $\dim H=1,$ then scalar-valued periodic pseudo-differential operators (periodic operators) have been investigated by several researchers over the past four decades. The  global quantisation of periodic operators using the Fourier transform, was introduced by  Volevich and Agranovich\cite{ag}  and further investigated in a broad context by Amosov \cite{Amosov}, Melo \cite{Melo}, McLean\cite{Mc}, Turunen and Vainikko\cite{tur} and, Ruzhansky and Turunen\cite{RT1,RT2,Ruz-2}. In these works, the authors developed the fundamental properties of the  pseudo-differential calculus  associated the with the H\"ormander classes as well as several applications to periodic PDE's and spectral theory. The generalisation of this approach to arbitrary  compact Lie groups can be found in the fundamental book \cite{Ruz} by Ruzhansky and Turunen. The analytical  and spectral properties for periodic operators have been treated in the work of Ruzhansky and Turunen \cite{Ruz-2,Ruz}, Delgado\cite{Profe},  Molahajloo and Wong\cite{s1,s2,m}, and  in the works of the authors \cite{Duvan2,Duvan3,Duvan4, CardonaBesov, CCK, Kumar, DasKumar, KS}. For the nuclearity of multilinear periodic operators in $L^p$-spaces we refer the reader to the works of the authors \cite{CardonaKumar20181,CardonaKumar20182} and to Delgado and Wong \cite{DW} for the linear case.
    \item If $\dim H=\aleph_0,$ the  references Amann \cite{Amann97}, Arendt and Bu \cite{Arendt0,Arendt}, Barraza, Gonzalez and Hern\'andez \cite{Bar}, Barraza,  Denk, Hern\'andez  and  Nau \cite{Bie}, Rabinovich \cite{Rabinovich}, Bu and Kim \cite{Bu,Bu2,Bu3} and Bu \cite{Bu4} have investigated the mapping properties of the vector-valued quantization in $L^p$-spaces and Besov spaces. Applications  to PDE's were studied by Denk and Nau \cite{Denk}, Keyantuo, Lizama, and  Poblete \cite{Keyantuo} and references therein. The classical aspects of the vector-valued Fourier analysis can be found in Bourgain \cite{Bou1}, Hieber \cite{Hieber} and K\"onig \cite{Konig}. For applications to probability theory we refer the reader to Bourgain \cite{Bou2}. 
\end{itemize}
Discrete vector-valued pseudo-differential operators (defined by \eqref{Discretevectorvalueddefinition}) will also be investigated in this work. For  the general properties of the symbolic calculus of discrete pseudo-differential operators on $\mathbb{Z}^n$ we refer to Botchway, Kibiti, and  Ruzhansky \cite{Botchway}. Also, the properties for these operators acting on $\ell^p(\mathbb{Z}^n)$ can be found in Molahajloo \cite{M}, Rodriguez \cite{rod} and the paper of the first author \cite{CardonaZn}. The nuclearity of Fourier integral operators on $\mathbb{R}^n,$ $\mathbb{Z}^n,$ compact Lie groups and compact homogeneous manifolds has been investigated in \cite{CardonaFIO}.

\subsection{Main results}\label{MainResults} In order to present our main result we need to fix some notations. First, we recall the definition of  $s$-nuclear operators on  Banach spaces due to Grothendieck \cite{Groth2, Groth1}. Let $E$ and $F$ be two Banach spaces. A compact linear operator $T:E \rightarrow F$ is called {\it $s$-nuclear}, $0<s \leq 1,$ if there exist sequences $\{g_n\}_n \subset E'$ and $\{h_n\}_n \subset F$ such that $ \sum_{n=1}^\infty \|g_n\|_{E'}^s\, \|h_n\|_{F}^s<\infty$ and for every $f \in E,$ we have \begin{equation}Tf= \sum_{n=1}^\infty \langle  g_n,f \rangle_{E',E} h_n.\end{equation}
	If $s=1$ we say that $T$ is a  nuclear operator. We denote by $\mathfrak{N}_s(E,F)$ the ideal of $s$-nuclear operators from $E$ to $F.$ The sequences $\{g_n\}_n \subset E'$ and $\{h_n\}_n \subset F$ are called a nuclear decomposition of $T.$ Grothendieck \cite{Groth2, Groth1} proved that the trace $\text{Tr}(T)$ is well defined for all nuclear operators if and only if the Banach space $E$ has the {\it approximation property}, i.e. for every compact set $K$ in $E$ and for every $\epsilon > 0$, there exists $F \in  \mathcal{F}(E)$ such that $\|x-Fx\|<\epsilon$, for all $x \in  K,$ where $\mathcal{F}(E)$ denotes the space of finite rank bounded linear operators
on $E$.  If $E=F$ is a Banach space satisfying the (Grothendieck) approximation property, we can associate to $T,$ a real number, called the nuclear trace of $T$ computed from a arbitrary nuclear decomposition of $T.$ Indeed, the nuclear trace of $T$ is defined by
	\begin{equation}
	    \textnormal{\bf{n-Tr(T)}}:=\sum_{n=1}^\infty \langle g_n, h_n \rangle_{E',E}.
	\end{equation} The condition that $E$ satisfies the Grothendieck approximation property implies that  $ \textnormal{\bf{n-Tr(T)}}$  is independent of the choice of the nuclear decomposition. Because we restrict our attention to separable Hilbert spaces, they satisfy the (Grothendieck) approximation property and consequently for all $1\leq p\leq \infty,$ $L^p(\mathbb{T}^n,H)$ has the approximation property (this is a particular situation of Proposition 7.11,  of Casazza \cite[p. 308]{Casazza}). If $E=F,$ and $T$ is a $s$-nuclear operator, $0<s\leq \frac{2}{3},$ it was shown by Grothendieck that the nuclear trace agrees with the spectral trace,
	\begin{equation}\label{Grothendieck}
	    \textnormal{\textbf{n-Tr}}(T)=\textnormal{\textbf{sp-Tr}}:=\sum_{n=1}^\infty \lambda_n(T).
	\end{equation}Here, $\lambda_n(T),$ $n\in\mathbb{N},$ denotes the sequence of eigenvalues of $T$ with multiplicities taken into account. The Grothendieck result \eqref{Grothendieck} will be useful in our further analysis for index theorems. Indeed, if $E$ is a Hilbert space and $T$ is a Fredholm operator (invertible modulo compact operators), in view of the celebrated McKean-Singer index formula
	\begin{equation}
	    \textnormal{\textbf{Ind}}(T)=\dim[ \textnormal{ker} T]- \dim[ \textnormal{ker} T^*]: =\textnormal{\textbf{sp-Tr}}(e^{-tT^*T})-\textnormal{\textbf{sp-Tr}}(e^{-tTT^*}),\,\,\,t>0,
	\end{equation} we  investigate the index of periodic pseudo-differential operators by using the notion of nuclear trace. 
	
	Now, we present our main results by starting with those regarding the nuclearity of periodic pseudo-differential  operators defined by  \eqref{Pseudotorus}.
	By using Delgado's theorem (Theorem \ref{DelTheorem2.8}) in Theorem \ref{chartn}
 we will prove that for $H=L^2(Y,\nu),$ and $0<s\leq 1,$ the operator $T_\sigma:L^p(\mathbb{T}^n, L^2(Y, \nu)) \rightarrow L^p(\mathbb{T}^n, L^2(Y, \nu)),$  $ 1 \leq p <\infty,$ extends to a $s$-nuclear operator, if and only if, the symbol $\sigma$ admits a decomposition of the form
 \begin{equation} \label{18vish'}
 \sigma(x,\xi)= e^{-2\pi i x \cdot \xi} \sum_{k=1}^\infty h_k \otimes \mathscr{F}(g_k) (x, - \xi)\textnormal{Id}_{L^2(Y,\nu)},\end{equation}
 where $\{g_k\}_k$ and $\{h_k\}_k$ are two sequences in $L^p(\mathbb{T}^n, L^2(Y, \nu))$ and  $L^{p'}(\mathbb{T}^n, L^2(Y, \nu)),$ where $p'$  the Lebesgue conjugate of $p,$ i.e., $\frac{1}{p}+\frac{1}{p'}=1,$ respectively such that $$\sum_{k=1}^\infty \|h_k\|^s_{L^{p'}(\mathbb{T}^n, L^2(Y, \nu))} \|g_k\|^s_{L^p(\mathbb{T}^n, L^2(Y, \nu))} <\infty,$$ and the convergence in \eqref{18vish'} will be understood in sense of the usual norm convergence  on $L^2(Y, \nu).$	
 
The presence of the exponential factor in  \eqref{18vish'} could suggest that the considered symbols do not satisfy the standard estimates of periodic H\"ormander classes. However, this equality cannot be used for extracting information of the derivatives a  symbol, if we focus our attention to smooth symbols, in turns we have not information about  the regularity of the functions $h_k$ and the discrete differences of the sequences $\mathscr{F}g_k,$ $k\in\mathbb{N}$. As in the case of scalar valued periodic operators, we only need information on the decaying of the symbol, in our case of the operator norm of the symbol $\Vert \sigma(x,\xi)\Vert_{\textnormal{op}}$ when $|\xi|\rightarrow \infty.$  Curiously, as a application of this characterisation, we will prove that for $1<p\leq 2,$     if $T_\sigma:L^{p}(\mathbb{T}^n,L^2(Y,\nu))\rightarrow L^{p}(\mathbb{T}^n,L^2(Y,\nu))$ is nuclear, then $\Vert a(x,\xi)\Vert_{\textnormal{op}}\in L^{p'}(\mathbb{T}^n\times \mathbb{Z}^n),$ this means that
 \begin{equation}
     \Vert \Vert  a(x,\xi)\Vert_{\textnormal{op}}\Vert_{ L^{p'}(\mathbb{T}^n\times \mathbb{Z}^n)}:=\left(\sum\limits_{\xi\in\mathbb{Z}^n} \int\limits_{\mathbb{T}^n}\Vert a(x,\xi)\Vert^{p'}_{\textnormal{op}}dx\right)^{\frac{1}{p'}}<\infty.
 \end{equation} 
 An analogue result can be proved for vector-valued  discrete operators. Indeed, if $\mathfrak{t}_a$ is formally defined by
 \begin{eqnarray}\label{Discretevectorvalueddefinition}
  \mathfrak{t}_af(x):=\int\limits_{\mathbb{T}^n}e^{i2\pi x\cdot \xi}\sigma(x,\xi)(\mathcal{F}_{\mathbb{Z}^n}f)(\xi)d\xi,\,\,
 \end{eqnarray} where $$\mathcal{F}_{\mathbb{Z}^n}f(\xi):=\sum_{x\in\mathbb{Z}^n}e^{-i2\pi x\cdot \xi }f(x),$$ is the discrete vector-valued Fourier transform defined on those discrete vector-valued functions $f:\mathbb{Z}^n\rightarrow H$ with compact support, we will prove in Theorem \ref{VishcharZn}, that $\mathfrak{t}_a:L^p(\mathbb{Z}^n, L^2(Y, \nu)) \rightarrow L^p(\mathbb{Z}^n, L^2(Y,\nu)), \, 1 \leq p < \infty$  is $s$-nuclear if and only if  
 \begin{equation} \label{18vish2'}
     a(x, \xi)= e^{-i2\pi x \cdot \xi} \sum_{k=1}^\infty (h_k \otimes \mathcal{F}_{\mathbb{Z}^n} g_k)(x, -\xi)\textnormal{Id}_{L^2(Y,\nu)},
 \end{equation}
 where $\{g_k\}_k$ and $\{h_k\}_k$ are two sequences in $L^p(\mathbb{Z}^n, L^2(Y, \nu))$ and  $L^{p'}(\mathbb{Z}^n, L^2(Y, \nu))$ respectively such that $$\sum_{k=1}^\infty \|h_k\|^s_{L^{p'}(\mathbb{Z}^n, L^2(Y, \nu))} \|g_k\|^s_{L^p(\mathbb{Z}^n, L^2(Y, \nu))} <\infty,$$ and the convergence in \eqref{18vish2'} will be understood in the sense of the usual norm convergence on $L^2(Y, \nu).$ Newly, as application of this discrete characterisation,   for $1<p\leq 2,$     if $\mathfrak{t}_\sigma:L^{p}(\mathbb{Z}^n,L^2(Y,\nu))\rightarrow L^{p}(\mathbb{Z}^n,L^2(Y,\nu))$ is nuclear, then $\Vert a(x,\xi)\Vert_{\textnormal{op}}\in L^{p'}(\mathbb{Z}^n\times \mathbb{T}^n),$ this means that
 \begin{equation}
     \Vert \Vert  a(x,\xi)\Vert_{\textnormal{op}}\Vert_{ L^{p'}(\mathbb{Z}^n\times \mathbb{T}^n),}^{p'}:= \int\limits_{\mathbb{T}^n}\sum\limits_{\xi\in\mathbb{Z}^n}\Vert a(x,\xi)\Vert^{p'}_{\textnormal{op}}dx<\infty.
 \end{equation} 
 
 In the general case, we provide sufficient conditions for the nuclearity of pseudo-differential operators. In this general setting we find criteria by fixing an arbitrary orthonormal basis  $$B=\{e_\alpha:\alpha \in I\},$$  of the Hilbert space $H.$ So, in Theorem \ref{TheoremNuclearSufficientCondition}, we prove that 
if $A:C^\infty(\mathbb{T}^n,H)\rightarrow C^\infty(\mathbb{T}^n,H)$ is a vector-valued periodic pseudo-differential operator with operator-valued symbol $\sigma:\mathbb{T}^n\times \mathbb{Z}^n\rightarrow \mathcal{L}(H),$ and   $1\leq p_1,p_2< \infty,$  under the summability condition
\begin{equation}\label{Hiponsymbol'}
\Vert\{ \Vert \sigma(x,\xi)e_\alpha\Vert^s_{L^{p_2}(\mathbb{T}^n_x,H)}  \}_{(\xi,\alpha)} \Vert_{\ell^1(\mathbb{Z}^n\times I)}:=    \sum_{\xi\in\mathbb{Z}^n}\sum_{\alpha\in I}\Vert \sigma(x,\xi)e_\alpha\Vert^s_{L^{p_2}(\mathbb{T}^n_x,H)}<\infty,
\end{equation}    the operator $A:L^{p_1}(\mathbb{T}^n,H)\rightarrow L^{p_2}(\mathbb{T}^n,H)$ extends to a $s$-nuclear operator. If $1\leq p<\infty$ and $p=p_1=p_2,$ in terms of the basis $B,$ the nuclear trace of $A$ is given by
\begin{equation}\label{nTraceformula1'}
    \textnormal{\textbf{n-Tr}}(A)=\sum_{\xi\in\mathbb{Z}^n}\sum_{\alpha\in I}\int\limits_{\mathbb{T}^n}\langle \sigma(x,\xi)e_\alpha,e_\alpha\rangle_{H}dx. 
\end{equation} In particular, if $A$ is a Fourier multiplier, the symbol $\sigma(\xi)$ belongs to the Schatten-von Neumann class $\mathfrak{S}_{s}(H)$ of order $s,$ and
\begin{equation}\label{nucleartracemultiplier'}
    \textnormal{\textbf{n-Tr}}(A)=\sum_{\xi\in\mathbb{Z}^n}\textnormal{\textbf{n-Tr}}(\sigma(\xi)).
\end{equation} 

However, in the general case, if we assume that the operators $\sigma(x,\xi),$ $\xi\in\mathbb{Z}^n,$ are simultaneously positive operators, the trace formula  \eqref{nTraceformula1'} is independent of the basis $B$ under consideration and 
\begin{equation}\label{nTraceformula1''}
    \textnormal{\textbf{n-Tr}}(A)=\sum_{\xi\in\mathbb{Z}^n}\int\limits_{\mathbb{T}^n}\textnormal{\textbf{n-Tr}}(\sigma(x,\xi))dx.
\end{equation}
As an application of these trace formulae,  every Fourier multiplier $A\equiv T_\sigma:L^2(\mathbb{T}^n,H)\rightarrow L^2(\mathbb{T}^n,H)$, (closed and densely defined) admitting a Fredholm extension and under certain conditions, has an analytical index satisfying the index formula
\begin{equation}\label{IndexFormulaNov2018'}
    \textnormal{\textbf{Ind}}(A)=\sum_{\eta\in \mathbb{Z}^n}\textnormal{\textbf{Ind}}(\sigma(\eta)),
\end{equation} which in turns implies that $\textnormal{\textbf{Ind}}(\sigma(\eta))=0$ except at finitely many integers $\eta.$
This will be proved in Theorem \ref{Indextheorem1}. Finally, as a consequence of this index theorem we study the index of certain periodic  elliptic operators (see Theorem \ref{periodicindex2018}).

\section{Preliminaries}
\subsection{Vector-valued Fourier transform and Fourier multipliers}
 If $f:\mathbb{T}^n\rightarrow H$ is a smooth function then its Fourier transform at $\eta\in\mathbb{Z}^n$ is defined by
\begin{equation}
(\mathscr{F}f)(\eta)\equiv \widehat{f}(\eta)=\int\limits_{\mathbb{T}^n}e^{-i2\pi x\cdot\eta}f(x)dx\in H,
\end{equation}
then a linear operator $A:C^{\infty}(\mathbb{T}^n, H)\rightarrow C^{\infty}(\mathbb{T}^n, H)$ is a Fourier multiplier if
\begin{eqnarray}\label{fourier}
(\mathscr{F}(Af))(\eta)\equiv \widehat{Af}(\eta)=\sigma(\eta)\widehat{f}(\eta),
\end{eqnarray}
for some operator-valued function $\sigma:\mathbb{Z}^n\rightarrow \mathcal{L}(H)$ from $\mathbb{Z}^n$ into the algebra of bounded linear operators on $H.$ 
In this case, $A$ can be formally written as
\begin{equation}\label{multiplier}
Af(x)\equiv T_\sigma f(x)=\sum_{\eta\in\mathbb{Z}^n}e^{i2\pi x\cdot\eta}\sigma(\eta)\widehat{f}(\eta),
\end{equation}
and the operator-valued function $\sigma$ is called the global symbol of the operator $A.$ There exists a one to one correspondence between a symbol (i.e, a function from $\mathbb{Z}^n$ into $\mathcal{L}(H)$) and its corresponding operator $A$ defined by  \eqref{multiplier}. This correspondence is the consequence of following formula
\begin{equation}\label{symbol}
\sigma(\xi)v=e_{-\xi}A(e_\eta\otimes  v), \,\,\,e_\xi(x):=e^{i2\pi x\cdot\xi},\,\, x\in\mathbb{T}^n,\,\xi\in\mathbb{Z}^n,\,\,\,v\in H,
\end{equation} which can be derived if, in \eqref{multiplier}, we replace $f$ by the function  $x \mapsto (e_{\xi} \otimes v)(x):= e^{i2\pi x\cdot \xi}v$. If $A\equiv T_\sigma:L^2(\mathbb{T}^n,H)\rightarrow L^2(\mathbb{T}^n,H)$ and $S\equiv T_\tau:L^2(\mathbb{T}^n,H)\rightarrow L^2(\mathbb{T}^n,H)$ are bounded (Fourier multipliers) on $L^2(\mathbb{T}^n,H),$  the properties of the vector-valued  Fourier transform implies that the operator $T_{a}:=A\circ S$ is a bounded (Fourier multiplier) on $L^2(\mathbb{T}^n,H),$ with symbol defined by
\begin{equation}\label{Fouriercomposition}
 a:= (\eta\mapsto   a(\eta):=\sigma(\eta)\circ \tau (\eta)):\mathbb{Z}^n\rightarrow \mathcal{L}(H).
\end{equation}
For applications, we also consider Fourier multipliers $A$ with   operator-valued symbols admitting unbounded values, this means that  for some $\eta,$ the (densely defined) operator $\sigma(\eta)$ is unbounded  on $H.$ We also require in this case a framework for the composition of these operators. For this, throughout the paper we fix an orthonormal  basis $B=\{e_\alpha:\alpha\in I \}$  of the Hilbert space $H$, we develop our analysis in terms of a such basis, but, later we will show that some geometric properties of the corresponding operators are independent of the considered basis $B$.

The subspace $ L^2_{\textnormal{fin}\,B}
(\mathbb{T}^n,H)
$ of $ L^2(\mathbb{T}^n,H)
$ is the subspace  consisting of the finite linear combinations of functions of the form $\exp_\xi(\cdot)\otimes e_{\alpha},$ $\exp_\xi(x):=e^{i2\pi x\cdot \xi},$ $x\in\mathbb{T}^n.$ If $A:L^2_{\textnormal{fin}\,B}(\mathbb{T}^n,H)\rightarrow L^2(\mathbb{T}^n,H)$ and $S:L^2_{\textnormal{fin}\,B}(\mathbb{T}^n,H)\rightarrow L^2(\mathbb{T}^n,H)$ have symbols admitting unbounded operator-valued values, but,  \begin{equation}\label{conditionsobre B}
    B\subset  \bigcap_{\eta\in\mathbb{Z}^n} \textnormal{Dom}(\sigma(\eta))\cap \textnormal{Dom}(\tau(\eta)),\,\,\textnormal{Rank}(\tau(\eta))\subset \textnormal{Dom}(\sigma(\eta)),\,\,\eta\in\mathbb{Z}^n,
\end{equation} then the composition operator $T_a=A\circ S$ is well defined from $L^2_{\textnormal{fin}\,B}(\mathbb{T}^n,H)$ into  $ L^2(\mathbb{T}^n,H)$. In this case, the symbol  of $T_a,$ $a$ is given by $a(\eta)=\sigma(\eta)\tau(\eta),$ $\eta\in\mathbb{Z}^n.$ Observe that for every $\eta,$ $a(\eta)$ is a (well-defined) densely defined operator  on $H.$ In the following lemma we summarise the previous analysis, some known results (that we present here for completeness), and some  spectral properties of Fourier multipliers. 

\begin{lemma}\label{elpapadeloslemma} Let $\sigma$ and $\tau$ be two  operator-valued symbols from $\mathbb{Z}^n$ into $H$. Then,
\begin{enumerate}
    \item if $A\equiv T_\sigma:L^2(\mathbb{T}^n,H)\rightarrow L^2(\mathbb{T}^n,H)$ and $S\equiv T_\tau:L^2(\mathbb{T}^n,H)\rightarrow L^2(\mathbb{T}^n,H)$ are bounded (Fourier multipliers) on $L^2(\mathbb{T}^n,H),$ then the (Fourier multiplier) operator $T_{a}:=A\circ S$ is a bounded operator on $L^2(\mathbb{T}^n,H)$ with symbol defined by \eqref{Fouriercomposition}.
    \item If $A\equiv T_\sigma:L^2_{\textnormal{fin}\,B}(\mathbb{T}^n,H)\rightarrow L^2(\mathbb{T}^n,H)$ and $S\equiv T_\tau:L^2_{\textnormal{fin}\,B}(\mathbb{T}^n,H)\rightarrow L^2(\mathbb{T}^n,H)$ have symbols satisfying \eqref{conditionsobre B}, then the periodic operator $T_a:L^2_{\textnormal{fin}\,B}(\mathbb{T}^n,H)\rightarrow L^2(\mathbb{T}^n,H)$ defined by $T_{a}:=A\circ S$ on $L^2_{\textnormal{fin}\,B}(\mathbb{T}^n,H)$  has symbol defined by \eqref{Fouriercomposition}.
    \item If $A\equiv T_\sigma:L^2_{\textnormal{fin}\,B}(\mathbb{T}^n,H)\rightarrow L^2(\mathbb{T}^n,H)$ is a Fourier multiplier and $\lambda\in \textnormal{Resolv}(A)$ (the resolvent set of $A$), then $\lambda-A=T_{\lambda-a},$ $\lambda\in \textnormal{Resolv}(\sigma(\eta))$ for all $\eta\in \mathbb{Z}^n,$ and the symbol of the resolvent operator $(\lambda-A)^{-1}$ is defined by
    \begin{eqnarray}
    R_\lambda(\sigma):=(\eta\mapsto  R_\lambda(\sigma(\eta)):(\lambda-\sigma(\eta))^{-1}): H\rightarrow H.
    \end{eqnarray}
    \item  The spectrum and the resolvent of $A:L^2_{\textnormal{fin}\,B}(\mathbb{T}^n,H)\subset L^2(\mathbb{T}^n,H) \rightarrow L^2(\mathbb{T}^n,H)$ are sets given by
    \begin{equation}
       \textnormal{Spect}(A)=\bigcup_{\eta\in\mathbb{Z}^n} \textnormal{Spect}(\sigma(\eta)),\,\,\,\,\,\textnormal{Resolv}(A)=\bigcap_{\eta\in\mathbb{Z}^n} \textnormal{Resolv}(\sigma(\eta)).
    \end{equation}
    \item If $A:L^2_{\textnormal{fin}\,B}(\mathbb{T}^n,H)\subset L^2(\mathbb{T}^n,H) \rightarrow L^2(\mathbb{T}^n,H)$ is a Fourier multiplier, the formal adjoint of $A,$ $A^*$ is a Fourier multiplier with symbol $\sigma^*$ defined  by 
    \begin{eqnarray}
    \sigma^*:=(\eta\mapsto \sigma^*(\eta):=\sigma(\eta)^*): H\rightarrow H,
    \end{eqnarray} where $\sigma^*(\eta)$ is the adjoint operator of the (closable) operator $\sigma(\eta)$. 
    \item Let us consider the heat
     semigroups $\{e^{-t\Delta_a}\}$ and $\{e^{-t\Delta_b}\}$  defined by the functional calculus associated with the Laplacians   $\Delta_a=A^*A$ and $\Delta_b=AA^*,$ where $A:L^2_{\textnormal{fin}\,B}(\mathbb{T}^n,H)\subset L^2(\mathbb{T}^n,H) \rightarrow L^2(\mathbb{T}^n,H),$ is a Fourier multiplier. Then, for all $t>0,$ $T_{a(t)}:=e^{-t\Delta_a}$ and $T_{b(t)}:=e^{-t\Delta_b}$ are Fourier multipliers with operator-valued symbols defined by
     \begin{equation}
         a(t):=(\eta\mapsto a(t)(\eta):=e^{-t\sigma(\eta)^*\sigma(\eta)}): H\rightarrow H,
     \end{equation} 
     \begin{equation}
         b(t):=(\eta\mapsto b(t)(\eta):=e^{-t\sigma(\eta)^*\sigma(\eta)}): H\rightarrow H.
     \end{equation}
\end{enumerate}
  \end{lemma}
  \begin{proof} 
The parts (1), (2) and (6)    can be proved by straightforward computation. In order to prove (3), we observe that $\lambda\in \textnormal{Resolv}(A)$ if and only $\lambda-A$ is an invertible operator on $L^2(\mathbb{T}^n,H)$ such that $\lambda-A$ and its inverse, $(\lambda-A)^{-1},$ are bounded operators. From the equality
$$ (\lambda-A)(\lambda-A)^{-1}=(\lambda-A)^{-1}(\lambda-A)=Id_{H}, $$ a standard argument via the Fourier transform allows us to conclude that, for every $\xi,$ $\lambda-\sigma(\xi)$ is a bounded and invertible operator on $H$ and the symbol of $(\lambda-A)^{-1}$ is defined by the sequence of operators $\{(\lambda-\sigma(\xi))^{-1}\}_{\xi\in \mathbb{Z}^n}.$ The property $(4)$ can be deduced from $(3).$ In order to prove $(5),$ let us observe that from the equality
\begin{eqnarray}
 \langle Af,g \rangle_{L^2(\mathbb{T}^n,H)}=\langle f,A^*g \rangle_{L^2(\mathbb{T}^n,H)},
\end{eqnarray}we deduce, using Plancherel theorem, 
\begin{eqnarray}
 \langle \sigma(\xi)\widehat{f},\widehat{g} \rangle_{\ell^2(\mathbb{Z}^n,H)}=\langle \widehat{f},\widehat{A^*g} \rangle_{\ell^2(\mathbb{Z}^n,H)}.
\end{eqnarray} Consequently, we have that $ \langle \widehat{f},\,\sigma(\xi)^*\widehat{g} \rangle_{\ell^2(\mathbb{Z}^n,H)}=\langle \widehat{f},\widehat{A^*g} \rangle_{\ell^2(\mathbb{Z}^n,H)}.$ So, $\sigma(\xi)^*\widehat{g}=\widehat{A^*g}.$ This shows that $A^*$ is a vector-valued Fourier multiplier and that its symbol is defined by the sequence $\{\sigma(\xi)^*\}_{\xi\in\mathbb{Z}^n}.$ So, we finish the proof. 
  \end{proof}
\subsection{Basics of tensor products and nuclear operators}
  We will present some basic concepts develped in \cite{Delgado} related to trace of a nuclear  operator and tensor products to make our paper self contained.
  
Let $E$ be a Banach space and let $\mathcal{L}(E)$ be the algebra of all bounded linear operator on $E.$ Let $\mathfrak{F}(E)$ denoted the ideal of $\mathcal{L}(E)$ consisting of finite rank operators from $E$ into $E.$ It is well known that every $T \in \mathfrak{F}(E)$ can be represented as $T= \sum_{i=1}^n f_i \otimes x_i$ where $f_i \in E'$ and $x_i \in E,$ $i =1,2, \ldots, n.$ We note that this representation is not unique.   The trace of a finite rank operator $T$ is given by 
\begin{equation}
\textnormal{\textbf{Tr}}(T):= \sum_{i=1}^n \langle f_i, x_i \rangle_{E',E}.\end{equation}  It can be easily seen that Trace is independent of the choice of the representation of $T.$ If $E$ has an approximation property, $\mathfrak{F}(E)$ is dense in $\mathfrak{N}(E).$ Also, every $T \in \mathfrak{N}(E)$ can be written as $T= \sum_{i=1}^\infty f_i \otimes x_i,$ where $f \in E',\, x \in E.$ Then the (nuclear) trace of a $T \in \mathfrak{N}(E)$ which is independent of the choice of the representation of $T$ is given by 
\begin{equation}\textnormal{\textbf{n-Tr}}(T) = \sum_{i=1}^\infty \langle f_i, x_i \rangle_{E', E}.\end{equation}
The nuclear trace also appear in the more general context of $s$-nuclear operators, $0<s\leq 1,$ as we have seen in the introduction.
Since in this paper we are dealing with vector valued $L^p$-spaces, we will study particularly the Banach space $E= L^p( \Omega, \mu, L^2(Y, \nu)),$  of all $L^2(Y, \nu)$-valued Bochner measurable functions such that $\|f(\cdot)\|_{L^2(Y,\nu)} \in L^p(\Omega, \mu), \, 1 \leq p <\infty,$ where $\mu$ and $\nu$ are $\sigma$-finite measures on $\Omega$ and $Y$ respectively.  We also assume that $H=L^2(Y, \nu)$ is a separable Hilbert space. 

We define the tensor product $f \otimes g$ of $f \in L^p(\mu, L^2(Y, \nu))$ and $g \in L^{p'}(\mu, L^2(Y, \nu))$ with $\frac{1}{p}+\frac{1}{p'}=1$ as follows: for $(x,y) \in \Omega \times \Omega$ and $(z,w) \in Y \times Y,$ we set 
   $$[(f\otimes g)(x,y)](z,y)=: f(x)(z) \,g(y)(w),$$
where we have denoted  $$g(y,w)=g(y)(w)\textnormal{   and   }f(x,z)=:f(x)(z).$$ Observe that for $a.e.w.\,\,(x,y)\in \Omega\times \Omega,$ the function $(f\otimes g)(x,y)\in L^2(\nu\otimes \nu, Y\times Y).$ Indeed,
\begin{eqnarray}
 \Vert (f\otimes g)(x,y) \Vert_{ L^2(\nu\otimes \nu, Y\times Y) }=\Vert f(x,\cdot) \Vert_{L^2(Y,\nu)} \Vert g(y,\cdot) \Vert_{L^2(Y,\nu)}. \end{eqnarray}
 As a consequence of the H\"older Inequality, we have $x\mapsto f\otimes g(x,x)\in L^2(Y,\nu)$ and consequently,
 \begin{eqnarray}
  \Vert f\otimes g(x,x) \Vert_{L^2(Y,\nu)}\leq \Vert f\Vert_{L^p(\Omega,L^2(Y,\nu))}\Vert g\Vert_{L^{p'}(\Omega,L^2(Y,\nu))}.
 \end{eqnarray}
 The inner product on $L^2(Y,\nu)$ induces to a duality $\langle\cdot \,, \cdot\rangle_{p,p'}$ between the spaces $L^p(\Omega,L^2(Y,\nu))$ and $L^{p'}(\Omega,L^2(Y,\nu))$  for $1/p+1/{p'}=1,$ which is defined by
 \begin{eqnarray}
  \langle f, g\rangle_{p,p'}:=\int\limits_{\Omega}\langle f(x,\cdot),\overline{g(x,\cdot)}\rangle_{L^2(Y,\nu)}d\mu(x).
 \end{eqnarray}
 This duality, allows us to understand the tensor product $f \otimes g$ as an integral operator. Indeed, it was noted in \cite[Lemma 2.6]{Delgado} that $f \otimes g$ defines an integral operator on $L^p(\mu, L^2(Y, \nu)).$ In fact, for any $h \in L^p(\mu , L^2(Y, \nu))$ we have 
$$ (f \otimes g)(h)(x)= \int\limits_\Omega[ K_{f \otimes g}(x,y)] h(y) d\mu(y) ,$$ where  $K_{f \otimes g}(x,y)$ denotes the integral kernel of the integral operator $f \otimes g$ and given by $K_{f \otimes g}(x,y)=(f \otimes g) (x,y).$ Here,
\begin{eqnarray}
 [ K_{f \otimes g}(x,y)] h(y):=g(x,\cdot)\int\limits_{Y}f(y,z)h(y,z)d\nu(z).
\end{eqnarray}
It can be checked that 
\begin{eqnarray}
 (f \otimes g)(h)(x)=\langle h,f \rangle_{p,p'}g(x,\cdot).
\end{eqnarray}
Lemma 2.7 in Delgado \cite{Delgado} establish that the trace of the tensor product operator $f \otimes g$ is given by 
$$ \textnormal{\textbf{Tr}}(f \otimes g)= \int\limits_\Omega \textnormal{\textbf{Tr}}(K_{f \otimes g}(x,x))\, d\mu(x).$$ The following theorem of Delgado \cite[Theorem 2.8]{Delgado} presents a  characterisation of $s$-nuclear, $0<s \leq 1,$ operators on $L^p(\mu, L^2(Y, \nu)).$ 
\begin{theorem} \label{DelTheorem2.8}
An operator $T: L^p(\mu, L^2(Y, \nu)) \rightarrow L^p(\mu, L^2(Y, \nu))$ is a $s$-nuclear, $0<s\leq 1, $ operator if and only if there exist two sequences $(f_n)_{n\in\mathbb{N}} \subset L^p(\mu, L^2(Y, \nu))$ and $(g_n)_{n\in \mathbb{N}} \subset L^{p'}(\mu, L^2(Y, \nu))$ such that $$\sum_{n=1}^\infty \|f_n\|^s_{L^p(\mu, L^2(Y, \nu))} \|g_n\|^s_{L^{p'}(\mu, L^2(Y, \nu))}< \infty$$ such that for all $f \in L^p(\mu, L^2(Y, \nu))$ 
$$Tf(x)= \int\limits_\Omega \left(\sum_{n=1}^\infty (f_n \otimes g_n)(x,y) \right) f(y) \, d\mu(y)\,\,\,\,\,\, a.e. \, x.$$ 
Consequently, if $T$ is a $s$-nuclear operator on $L^p(\mu, L^2(Y, \nu))$ then $T$ posses a kernel $$K(x,y) = \sum_{n=1}^\infty (f_n \otimes g_n)(x,y)$$ with $f_n,\, g_n$ as above.
\end{theorem} 

\section{Nuclear vector-valued periodic and discrete pseudo-differential  operators} \label{nuche}

In this section we give a  characterisation of periodic and discrete pseudo-differential operators acting on the vector-valued $L^p$-spaces $L^p(\mathbb{T}^n, L^2(Y, \nu))$ and $L^p(\mathbb{Z}^n, L^2(Y, \nu))$  respectively.  To accomplish this objective, for $h\in L^{p'}(\mathbb{T}^n,L^2(Y,\nu))$ and $g\in L^p(\mathbb{T}^n,L^2(Y,\nu)),$ we define the function 
\begin{eqnarray}
 h\otimes \mathscr{F}(g)(x,\xi)=h(x)\mathscr{F}(g)(\xi),\,\,(x,\xi)\in \mathbb{T}^n\times \mathbb{Z}^n.
\end{eqnarray}Here, we keep the notation $h(x)=h(x,\cdot),$ $g(x)=g(x,\cdot),$ and we use,
\begin{eqnarray}
 \mathscr{F}(g)(\xi)=\int \limits_{\mathbb{T}^n}e^{-i2\pi x\cdot \xi}g(x)dx,
\end{eqnarray}
where the integral is understood in the Bochner sense. Note that we can also think $h \otimes \mathscr{F}(g)$ as an operator in the view of the fact that $h \otimes g$ can be viewed as an integral operator.  In this section, every   pseudo-differential operator \begin{eqnarray}  \label{Pseudotorus2Y}
Af(x)\equiv T_\sigma f(x):=\sum_{\xi\in\mathbb{Z}^n}e^{i2\pi x\cdot \xi}\sigma(x,\xi)\widehat{f}(\xi), \,\,x:=(x_1,x_2,\cdots ,x_n)\in \mathbb{T}^n,
\end{eqnarray} for all $f\in C^{\infty}(\mathbb{T}^n,H), $  will be considered admitting a bounded extension on  $L^p(\mathbb{T}^n,L^2(Y)),$ and for every $(x,\xi),$ we will assume that 
\begin{equation*}
    \sigma(x,\xi):L^2(Y)\rightarrow L^2(Y),
\end{equation*} extends to a bounded linear operator. In the case $1<p<\infty,$ these kind of operators arise from the Marcinkiewicz type condition
\begin{equation*}
    \Vert \partial_{x}^\alpha\Delta_{\xi}^\beta a(x,\xi) \Vert_{L^2(Y)}\leq (1+|\xi|)^{-|\beta|},\,\,\,\alpha,\beta\in \mathbb{N}_0^n,
\end{equation*}where $\Delta_\xi$ is the standard difference operator acting on vector-valued sequences.

\subsection{Characterisation of nuclear operators on $L^p(\mathbb{T}^n,L^2(Y,\nu))$} 

\begin{theorem} \label{chartn}
 Let $0<s\leq 1$ and let $T_\sigma:L^p(\mathbb{T}^n, L^2(Y, \nu)) \rightarrow L^p(\mathbb{T}^n, L^2(Y, \nu)), 1 \leq p <\infty$ be a vector-valued periodic pseudo-differential operator with operator valued symbol $\sigma: \mathbb{T}^n \times \mathbb{Z}^n \rightarrow \mathcal{L}(L^2(Y,\nu)).$ Then $T_\sigma$ is a $s$-nuclear, $0<s \leq 1,$ operator if and only if 
 
 \begin{equation} \label{18vish}
 \sigma(x,\xi)= e^{-2\pi i x \cdot \xi} \left(\sum_{k=1}^\infty (h_k \otimes \widehat{g_k}) (x, - \xi)\right),\,\,\,\,\,\text{a.e.}\,\, x \in \mathbb{T}^n, \xi \in \mathbb{Z}^n ,\end{equation}
 where $\{g_k\}_k$ and $\{h_k\}_k$ are two sequences in $L^p(\mathbb{T}^n, L^2(Y, \nu))$ and  $L^{p'}(\mathbb{T}^n, L^2(Y, \nu))$ respectively such that $$\sum_{k=1}^\infty \|h_k\|^s_{L^{p'}(\mathbb{T}^n, L^2(Y, \nu))} \|g_k\|^s_{L^p(\mathbb{T}^n, L^2(Y, \nu))} <\infty$$ and the convergence in right hand side of \eqref{18vish} will be understood in sense of the usual operator norm convergence in $\mathcal{L}(L^2(Y, \nu))$.   
\end{theorem}
\begin{proof}
 Let $T_\sigma$ be an $s$-nuclear, $0<s \leq 1$ operator. Therefore,  by \cite[Theorem 2.8]{Delgado}, there exist two sequences $\{g_k\}_k$ and $\{h_k\}_k$ in $L^p(\mathbb{T}^n, L^2(Y, \nu))$ and  $L^{p'}(\mathbb{T}^n, L^2(Y, \nu))$ respectively with $$\sum_{k=1}^\infty \|h_k\|^s_{L^{p'}(\mathbb{T}^n, L^2(Y, \nu))} \|g_k\|^s_{L^p(\mathbb{T}^n, L^2(Y, \nu))} <\infty$$ such that 
 
 \begin{equation}
    \,\,\,\,\,\,\,\,\,\,\,\,\,\,\,\,\,\,\,\,\,\,\,\,\,\,\,\,\,\,\,\,\,\,\,\,\,\,\,\,\,\,\,\,\,\,\,\,\,\,\,\,\,\,\,\, T_\sigma f(x)= \int \limits_{\mathbb{T}^n} \left( \sum_{k=1}^\infty (h_k \otimes g_k ) (x,y) \right) f(y)\, dy \,\,\,\,\,\,\, a.e.\, x
 \end{equation}
 Therefore, by the definition of $T_\sigma,$ we get 
 
 \begin{equation} \label{eqvis1}
     \sum_{\eta \in \mathbb{Z}^n} e^{2 \pi i x \cdot \eta} \sigma(x, \eta) \widehat{f}(\eta) =  \int \limits_{\mathbb{T}^n} \left( \sum_{k=1}^\infty (h_k \otimes g_k ) (x,y) \right) f(y)\, dy.
 \end{equation}
 
 Now, for every $\xi \in \mathbb{Z}^n$ and any $v \in L^2(Y, \nu)$ define a function $f$ by $f_\xi(y)= e^{2 \pi i \xi \cdot y} v \in L^2(Y, \nu)$ for every $y \in \mathbb{T}^n.$ Note that 
 $$\widehat{f_\xi}(\eta)= \begin{cases} v, & \xi= \eta, \\ 0, & \xi \neq \eta. \end{cases} \,\,\,\,\,\,\,\,\,\,\,$$
 Since equality \eqref{eqvis1} above holds for every $f \in L^p(\mathbb{T}^n, L^2(Y,\nu)).$ In particular, by choosing $f= f_\xi,$ $v\neq 0,$ we get
 
 \begin{align} \label{intechange}
   \,\,\,\,\,\,\,\,\,\,\,\,\,\,\,\,\,\,\,\,\,\,\,\,\,\,\,\,\,  e^{2 \pi i x \cdot \xi} \sigma(x, \xi)v &= \int \limits_{\mathbb{T}^n} \left( \sum_{k=1}^\infty (h_k \otimes g_k ) (x,y) \right) e^{2 \pi i \xi \cdot y} v \, dy \nonumber % \\&= \int \limits_{\mathbb{T}^n} \left( \sum_{k=1}^\infty h_k(x) g_k(y) \right) e^{2 \pi i \xi \cdot y} v \, dy. 
   \end{align}
   Now note that, 
   \begin{align*}
       \left\| \left( \sum_{k=1}^\infty (h_k \otimes g_k ) (x,y) \right) e^{2 \pi i \xi \cdot y}  \right\|_{L^2(\nu \times \nu)} &\leq \left\| \left( \sum_{k=1}^\infty (h_k \otimes g_k ) (x,y) \right)  \right\|_{L^2(\nu \times \nu)}  \\&= \left\| K(x, y)  \right\|_{L^2(\nu \times \nu)}.
   \end{align*} 
   In the view of \cite[Remark 2.11 (ii)]{Delgado} we see that $K \in L^1(\mathbb{T}^{n} \times \mathbb{T}^n, L^2(Y \times Y, \nu \times \nu))$ and therefore, $\left\| K(x, y)  \right\|_{L^2(\nu \times \nu)} \in L^1(\mathbb{T}^n \times \mathbb{T}^n)$. So, we can apply Dominated convergence theorem for the Bochner integral  to get 
   %\begin{align*}
    %  \sum_{k=1}^\infty \int_{\mathbb{T}^n}  \|g_k(y)\|_{L^2(Y, \nu)} \,dy &\leq  \left(\int_{\mathbb{T}^n} (\sum_{k=1}^\infty \|g_k(y)\|_{L^2(Y, \nu)})^{p'} \,dy \right)^{\frac{1}{p'}} \\& \leq C \left(\int_{\mathbb{T}^n} \sum_{k=1}^\infty \|g_k(y)\|_{L^2(Y, \nu)}^{p'} dy \right)^{\frac{1}{p'}} \\&= C \left( \sum_{k=1}^\infty \int_{\mathbb{T}^n} \|g_k(y)\|_{L^2(Y, \nu)}^{p'} dy \right)^{\frac{1}{p'}} \\&= C \|g_k\|_{L^{p'}(\mathbb{T}^n, L^2(Y, \nu))} <\infty,
   %\end{align*} where in the second last inequality follows from the Fubini theorem which can be applied as  $$\int_{\mathbb{T}^n} \sum_{k=1}^\infty \|g_k(y)\|_{L^2(Y, \nu)} \,dy =\sum_{k=1}^\infty \int_{\mathbb{T}^n} \|g_k(y)\|_{L^2(Y, \nu)}^{p'} dy <\infty.$$
   %Consequently, we can now interchange the summation and integration in \eqref{intechange} to get \lim_{l \rightarrow 1}\int \limits_{\{ y \in \mathbb{T}^n: |y_i| \leq l \quad \forall i = 1,2,\dots n\}}
   \begin{align*}
   \int \limits_{\mathbb{T}^n} \left( \sum_{k=1}^\infty (h_k \otimes g_k ) (x,y) \right) e^{2 \pi i \xi \cdot y} v \, dy &= \int \limits_{\mathbb{T}^n} \lim_{N \rightarrow \infty} \left( \sum_{k=1}^N (h_k \otimes g_k ) (x,y) \right) e^{2 \pi i \xi \cdot y} v \, dy \\&= \lim_{N \rightarrow \infty} \int \limits_{\mathbb{T}^n}  \left( \sum_{k=1}^N (h_k \otimes g_k ) (x,y) \right) e^{2 \pi i \xi \cdot y} v \, dy\\&= \lim_{N \rightarrow \infty}   \sum_{k=1}^N h_k(x) \int \limits_{\mathbb{T}^n}   g_k(y)   e^{2 \pi i \xi \cdot y} v \, dy \\&= \sum_{k=1}^\infty h_k(x)  \int \limits_{\mathbb{T}^n} g_k(y) e^{2 \pi i \eta \cdot y} dy\, v\\ &= \sum_{k=1}^\infty h_k(x) \, \widehat{g_k}(-\xi)v = \left( \sum_{k=1}^\infty (h_k \otimes \widehat{g_k})(x, -\xi) \right) v,
 \end{align*} for all $v \in L^2(Y, \nu).$
Therefore, we get 

\begin{equation}\sigma(x, \xi) = e^{-2 \pi i x \cdot \xi}\sum_{k=1}^\infty (h_k \otimes \widehat{g_k})(x, -\xi)\,\,\,\,\,\text{a.e.}\,\, x \in \mathbb{T}^n, \xi \in \mathbb{Z}^n.\end{equation}
 Conversely, assume that there exist two sequences $\{g_k\}_k$ and $\{h_k\}_k$ in $L^p(\mathbb{T}^n, L^2(Y, \nu))$ and $L^{p'}(\mathbb{T}^n, L^2(Y, \nu))$ with 
 $$\sum_{k=1}^\infty \|h_k\|^s_{L^{p'}(\mathbb{T}^n, L^2(Y, \nu))} \|g_k\|^s_{L^p(\mathbb{T}^n, L^2(Y, \nu))} <\infty$$ such that
 
 $$\sigma(x, \xi) = e^{-2 \pi i x \cdot \xi}\sum_{k=1}^\infty (h_k \otimes \mathscr{F}(g_k)(x, -\xi).$$ 
 
 Therefore, for any $f \in L^p(\mathbb{T}^n, L^2(Y, \nu))$
 \begin{align*}
  \,\,\,\,\,\,\,\,\,\,\,\,\,\,\,\,\,\,\,\,\,\,\,\,\,\,\,\,\,\, \,\,\,\,\,\,\,\,\,\,\,\,\,\,\,\,\,\,  T_\sigma f(x) &= \sum_{\eta \in \mathbb{Z}^n} e^{2 \pi i x \cdot \eta} \sigma(x, \eta) \widehat{f}(\eta) \\& = \sum_{\eta \in \mathbb{Z}^n} e^{2 \pi i x \cdot \eta} \left( e^{-2 \pi i x \cdot \eta}\sum_{k=1}^\infty (h_k \otimes \widehat{g_k}(x, -\eta) \right) \widehat{f}(\eta).
  \end{align*} Because $\sigma_N(x,\xi):=e^{-2 \pi i x \cdot \eta}\sum_{k\leq N} (h_k \otimes \widehat{g_k}(x, -\eta)$ converges to $\sigma(x,\xi)$ in the operator norm, we have
  \begin{align*}
    \sum_{\eta \in \mathbb{Z}^n}  \left( \sum_{k=1}^\infty (h_k \otimes \widehat{g_k}(x, -\eta) \right) \widehat{f}(\eta) &= \sum_{\eta \in \mathbb{Z}^n} \lim_{N\rightarrow\infty}\sum_{k\leq N} (h_k \otimes \widehat{g_k}(x, -\eta)\widehat{f}(\eta)\\ &=
    \lim_{N\rightarrow\infty}\lim_{M\rightarrow\infty}\sum_{|\eta|\leq M } \sum_{k\leq N} (h_k \otimes \widehat{g_k}(x, -\eta)\widehat{f}(\eta)\\
    &=\lim_{N\rightarrow\infty}\lim_{M\rightarrow\infty}\sum_{k\leq N}\sum_{|\eta|\leq M }  (h_k \otimes \widehat{g_k}(x, -\eta)\widehat{f}(\eta)\\
    &=\lim_{N\rightarrow\infty}\lim_{M\rightarrow\infty}\sum_{k\leq N}\sum_{|\eta|\leq M }  (h_k \otimes \widehat{g_k}(x, -\eta)\widehat{f}(\eta) \\
    &=\sum_{k=1}^\infty\sum_{\eta\in \mathbb{Z}^n}(h_k \otimes \widehat{g_k}(x, -\eta) \widehat{f}(\eta),
  \end{align*} where we have use that, for every $N,$ $\sigma_{N}$ is bounded on $L^2(Y).$ So, we have, by using dominated convergence theorem,  that
  \begin{align*}
\sum_{\eta \in \mathbb{Z}^n}  \left( \sum_{k=1}^\infty (h_k \otimes \widehat{g_k}(x, -\eta) \right) \widehat{f}(\eta)&=\sum_{k=1}^\infty \left( \sum_{\eta\in \mathbb{Z}^n}(h_k \otimes \widehat{g_k}(x, -\eta) \widehat{f}(\eta) \right) \\&=    \sum_{k=1}^\infty \left( \sum_{\eta \in \mathbb{Z}^n} \left( \int \limits_{\mathbb{T}^n} (h_k \otimes g_k)(x, y) e^{2 \pi i \eta \cdot y}\, dy \right)  \widehat{f}(\eta) \right) \\&= \sum_{k=1}^\infty \left( \lim_{M \rightarrow \infty} \sum_{|\eta| \leq M}  \int \limits_{\mathbb{T}^n} (h_k \otimes g_k)(x, y) e^{2 \pi i \eta \cdot y}\, dy   \widehat{f}(\eta) \right) \\&= \sum_{k=1}^\infty \left(   \int \limits_{\mathbb{T}^n} \lim_{M \rightarrow \infty}  \sum_{|\eta| \leq M} (h_k \otimes g_k)(x, y) e^{2 \pi i \eta \cdot y}\, dy   \widehat{f}(\eta) \right) \\&= \sum_{k=1}^\infty  \int \limits_{\mathbb{T}^n} h_k(x) g_k(y)  dy \left( \sum_{\eta \in \mathbb{Z}^n} e^{2 \pi i \eta \cdot y} \widehat{f}(\eta) \right)  \\
     &= \int \limits_{\mathbb{T}^n} \left( \sum_{k=1}^\infty (h_k \otimes g_k) (x,y) \right) f(y)\, dy,
 \end{align*} where the interchange of summation and integration in the last inequality can be justified using similar argument as in first part.
 Therefore, by \cite[Theorem 2.8]{Delgado}, $T_\sigma$ is an $s$-nuclear, $0<s \leq 1$ operator.
\end{proof}  

\begin{corollary} Let $T_\sigma:L^p(\mathbb{T}^n, L^2(Y, \nu)) \rightarrow L^p(\mathbb{T}^n, L^2(Y, \nu)), 1 \leq p <\infty$ be a vector-valued periodic pseudo-differential operator with operator valued symbol $\sigma: \mathbb{T}^n \times \mathbb{Z}^n \rightarrow \mathcal{L}(L^2(Y,\nu)).$ If $T_\sigma$ is a nuclear  operator then $\sigma(x, \xi)$ is a nuclear (trace class) operator for almost every $(x, \xi) \in \mathbb{T}^n \times \mathbb{Z}^n.$
\end{corollary}
\begin{proof}  
Since $T_\sigma$ is a nuclear operator, by the proof of Theorem \ref{chartn}, the operator $T_{\sigma}$ is given by  
    \begin{equation} 
 T_\sigma f(x) =\int \limits_{\mathbb{T}^n} \left( \sum_{k=1}^\infty (h_k \otimes g_k) (x,y) \right) f(y)\, dy,\,\end{equation}
 where the symbol can be expressed as,
\begin{equation} \sigma(x,\xi)v= e^{-2\pi i x \cdot \xi} \left(\sum_{k=1}^\infty h_k \otimes \mathscr{F}(g_k) (x, - \xi)\right)v,\end{equation}
 $v\in L^2(Y,\nu),$ with kernel $k=K(x,y)= \sum_{k=1}^\infty (h_k \otimes g_k) (x,y),$ 
 where $\{g_k\}_k$ and $\{h_k\}_k$ are two sequences in $L^p(\mathbb{T}^n, L^2(Y, \nu))$ and  $L^{p'}(\mathbb{T}^n, L^2(Y, \nu))$ respectively such that $$\sum_{k=1}^\infty \|h_k\|_{L^{p'}(\mathbb{T}^n, L^2(Y, \nu))} \|g_k\|_{L^p(\mathbb{T}^n, L^2(Y, \nu))} <\infty.$$
 Let us note that
 \begin{align*}
     \Vert \widehat{g}_{k}(-\xi) \Vert_{L^2(Y,\nu)} &=\Vert \int \limits_{\mathbb{T}^n}e^{i2\pi y\cdot \xi}g(y)dy\Vert_{L^2(Y,\nu)}\\&\leq \int \limits_{\mathbb{T}^n}\Vert g_k(y)\Vert_{L^2(Y,\nu)}dy\\
     &\leq  \left(   \int \limits_{\mathbb{T}^n}\Vert g_k(y)\Vert^p_{L^2(Y,\nu)}dy \right)^{\frac{1}{p}}=\Vert g_k \Vert_{L^p(\mathbb{T}^n,L^2(Y,\nu))}.\,\,\,\,\,\,\,\,\,\,\,\,\,\,\,\,\,\,\,\,\,\,\,\,\,\,\,\,\,\,\,\,\,\,\,\,\,\,\,\,\,\,\,\,\,\,\,\,
 \end{align*} Now, we will prove that the set 
 \begin{eqnarray}
 Z:=\{x\in \mathbb{T}^n: \sum_{k=1}^\infty  \Vert h_k(x) \Vert_{L^2(Y,\nu)}\Vert g_k \Vert_{L^p(\mathbb{T}^n,L^2(Y,\nu))}=\infty \}
 \end{eqnarray} has measure zero. This can be done by contradiction. If we assume that $Z$ has positive measure, then 
 \begin{equation}
     \int\limits_{Z}\left( \sum_{k=1}^\infty \Vert h_k(x) \Vert_{L^2(Y,\nu)}\Vert g_k \Vert_{L^p(\mathbb{T}^n,L^2(Y,\nu))} \right)dx=\infty.
 \end{equation} However, we have
 \begin{align*}
    & \int\limits_{Z}\left( \sum_{k=1}^\infty \Vert h_k(x) \Vert_{L^2(Y,\nu)}\Vert g_k \Vert_{L^p(\mathbb{T}^n,L^2(Y,\nu))} \right)dx\,\,\,\,\,\,\,\,\,\,\,\,\,\,\,\,\, \\ 
    &\leq \int \limits_{\mathbb{T}^n}\left( \sum_{k=1}^\infty \Vert h_k(x) \Vert_{L^2(Y,\nu)}\Vert g_k \Vert_{L^p(\mathbb{T}^n,L^2(Y,\nu))} \right)dx\\
    &= \sum_{k=1}^\infty \int \limits_{\mathbb{T}^n}\Vert h_k(x) \Vert_{L^2(Y,\nu)}dx\Vert g_k \Vert_{L^p(\mathbb{T}^n,L^2(Y,\nu))} \\
    &\leq  \sum_{k=1}^\infty \Vert h_k \Vert_{L^{p'}(\mathbb{T}^n,L^2(Y,\nu))}\Vert g_k \Vert_{L^p(\mathbb{T}^n,L^2(Y,\nu))}<\infty,
 \end{align*} and we obtain a contradiction. Now, by using that
 $$\sum_{k=1}^\infty  \Vert h_k(x) \Vert_{L^2(Y,\nu)}\Vert g_k \Vert_{L^p(\mathbb{T}^n,L^2(Y,\nu))}<\infty,\,\,a.e.x,\,\,\,\,$$ 
 and the inequality: $ \Vert \widehat{g}_{k}(-\xi) \Vert_{L^2(Y,\nu)} \leq \Vert g_k \Vert_{L^p(\mathbb{T}^n,L^2(Y,\nu))},$ we conclude that
 \begin{eqnarray}
 \sum_{k=1}^\infty  \Vert h_k(x) \Vert_{L^2(Y,\nu)}\Vert \widehat{g_k}(-\xi) \Vert_{L^2(Y,\nu)}<\infty,\,a.e.x.
 \end{eqnarray} So, we have proved that in the right hand side of the equation,
 \begin{eqnarray}
 \sigma(x,\xi)= e^{-2\pi i x \cdot \xi} \sum_{k=1}^\infty h_k \otimes \mathscr{F}(g_k) (x, - \xi)= \sum_{k=1}^\infty e^{-2\pi i x \cdot \xi}h_k(x)  \mathscr{F}(g_k) ( - \xi)
 \end{eqnarray}  we have a nuclear decomposition of $\sigma(x,\xi).$ So, we conclude that $\sigma(x,\xi)$ is a nuclear (trace class) operator on $L^2(Y,\nu).$
 \end{proof}

Now, we present the following application of Theorem  \ref{chartn}.
 
\begin{theorem}\label{decaying}
  Let $1<p\leq 2,$     and let $T_\sigma$ be the vector-valued operator  associated to $\sigma(\cdot,\cdot).$ If $T_\sigma:L^{p}(\mathbb{T}^n,L^2(Y,\nu))\rightarrow L^{p}(\mathbb{T}^n,L^2(Y,\nu))$ is nuclear. Then $\Vert a(x,\xi)\Vert_{\textnormal{op}}\in L^{p'}(\mathbb{T}^n\times \mathbb{Z}^n),$ this means that
 \begin{equation}
     \Vert \Vert  a(x,\xi)\Vert_{\textnormal{op}}\Vert_{ L^{p'}(\mathbb{T}^n\times \mathbb{Z}^n)}:=\left(\sum\limits_{\xi\in\mathbb{Z}^n} \int\limits_{\mathbb{T}^n}\Vert a(x,\xi)\Vert^{p'}_{\textnormal{op}}dx\right)^{\frac{1}{p'}}<\infty.
 \end{equation} 
\end{theorem}
\begin{proof}
 Let $1<p\leq 2.$   If $T_\sigma:L^{p}(\mathbb{T}^n,L^2(Y,\nu))\rightarrow L^{p}(\mathbb{T}^n,L^2(Y,\nu))$ extends to a  nuclear operator, by Theorem \ref{chartn} 
$$\sigma(x, \xi) = e^{-2 \pi i x \cdot \xi}\sum_{k=1}^\infty (h_k \otimes \mathscr{F}(g_k)(x, -\xi)\textnormal{Id}_{L^2(Y,\nu)},
$$ where   $\{g_k\}_{k\in\mathbb{N}}$ and $\{h_k\}_{k\in\mathbb{N}}$ are sequences of functions satisfying 
\begin{equation}
\sum_{k=0}^{\infty}\Vert g_k\Vert_{L^{p}(\mathbb{T}^n,L^2(Y,\nu))}\Vert h_{k}\Vert_{L^{p'}(\mathbb{T}^n,L^2(Y,\nu))}<\infty.
\end{equation} 
Therefore, if we take the operator norm, by the triangle inequality, we deduce the estimate
\begin{align*}
   \Vert  \sigma(x, \xi)\Vert_{\textnormal{op}} \leq  \sum_{k=1}^\infty |h_k(x)| |\mathscr{F}(g_k)( -\xi)|.
\end{align*}
Again, if we take the $L^{p'}_x$-norm, we have,
\begin{align*}
\Vert a(x,\xi) \Vert_{L^{p'}_x(\mathbb{T}^n,L^2(Y,\nu))} &= \left\Vert e^{-i2\pi x\cdot \xi}\sum_{k=1}^{\infty}h_{k}(x)\mathscr{F}(g_k)( -\xi)  \right\Vert_{L^{p'}_x(\mathbb{T}^n,L^2(Y,\nu))}\\
&=\left\Vert \sum_{k=1}^{\infty}h_{k}(x)\mathscr{F}(g_k)( -\xi)  \right\Vert_{L^{p'}_x(\mathbb{T}^n,L^2(Y,\nu))}\\
&\leq \sum_{k=1}^{\infty}\Vert h_{k}\Vert_{L^{p'}((\mathbb{T}^n,L^2(Y,\nu)))}|\mathscr{F}(g_k)( -\xi) |.
\end{align*} By the vector-valued  Hausdorff-Young inequality, we deduce that,  $$\Vert \mathscr{F}(g_k) \Vert_{L^{p'}(\mathbb{Z}^n,L^2(Y,\nu)))}\leq \Vert  {g}_k\Vert_{L^{p}(\mathbb{T}^n,L^2(Y,\nu))},$$ for all $1<p\leq 2.$
Consequently,
\begin{align*}
    \Vert \Vert a(x,\xi) \Vert_{\textnormal{op}}\Vert_{ L^{p'}_xL^{p'}_\xi(\mathbb{T}^n\times \mathbb{Z}^n),} &=\left(\sum\limits_{\xi \in\mathbb{Z}^n} \left(\int\limits_{\mathbb{T}^n}\|\sigma(x,\xi)\|_{\textnormal{op}}^{p'}dx\right)^{\frac{p'}{p'}}  \right)^{\frac{1}{p'}}\\
    &\leq\left\Vert  \sum_{k=1}^{\infty}\Vert h_{k}\Vert_{L^{p'}}|\mathscr{F}(g_k)( -\xi)| \right\Vert_{L^{p'}_\xi} \\
    &\leq \sum_{k=1}^{\infty}\Vert h_{k}\Vert_{L^{p'}}\Vert\mathscr{F}(g_k)\Vert_{L^{p'}}\\
     &\leq \sum_{k=1}^{\infty}\Vert h_{k}\Vert_{L^{p'}}\Vert{g}_k\Vert_{L^{p}}<\infty.
\end{align*}
Thus, we finish the proof.
\end{proof}

 \subsection{ Characterisation of nuclear operators on $L^p(\mathbb{Z}^n, L^p(Y, \nu))$}
 
 In order to characterise those nuclear operators on $L^p(\mathbb{Z}^n, L^p(Y, \nu)) $, for $h\in L^{p'}(\mathbb{Z}^n,L^2(Y,\nu))$ and $g\in L^p(\mathbb{Z}^n,L^2(Y,\nu)),$ we define the function 
\begin{eqnarray}
 h\otimes \mathcal{F}_{\mathbb{Z}^n}(g)(x,\xi)=h(x)\mathcal{F}_{\mathbb{Z}^n}(g)(\xi),\,\,(x,\xi)\in \mathbb{Z}^n\times \mathbb{T}^n.
\end{eqnarray} As in the previous subsection, we use the notation $h(x)=h(x,\cdot),$ $g(x)=g(x,\cdot),$ and we define,
\begin{eqnarray}
 \mathcal{F}_{\mathbb{Z}^n}(g)(\xi)=\sum_{x\in \mathbb{Z}^n}e^{-i2\pi x\cdot \xi}g(x).
\end{eqnarray}

 \begin{theorem}\label{VishcharZn} Let $0 < s \leq 1$ and let $\mathfrak{t}_a:L^p(\mathbb{Z}^n, L^2(Y, \nu)) \rightarrow L^p(\mathbb{Z}^n, L^2(Y,\nu)), \, 1 \leq p < \infty$ be a vector-valued discrete pseudo-differential operators associated with the operator valued symbol $a: \mathbb{Z}^n \times \mathbb{T}^n \rightarrow \mathcal{L}(L^2(Y, \nu )).$ Then $\mathfrak{t}_a$ is $s$-nuclear if and only if  
 \begin{equation} \label{18vish2}
     a(x, \xi)= e^{-i2\pi x \cdot \xi}\left( \sum_{k} (h_k \otimes \mathcal{F}_{\mathbb{Z}^n} g_k)(x, -\xi)\right)\,\,\, \text{a.e.}\,\, x \in \mathbb{Z}^n, \xi \in \mathbb{T}^n,,
 \end{equation}
 where $\{g_k\}_k$ and $\{h_k\}_k$ are two sequences in $L^p(\mathbb{Z}^n, L^2(Y, \nu))$ and  $L^{p'}(\mathbb{Z}^n, L^2(Y, \nu))$ respectively such that $\sum_{k=1}^\infty \|h_k\|^s_{L^{p'}(\mathbb{Z}^n, L^2(Y, \nu))} \|g_k\|^s_{L^p(\mathbb{Z}^n, L^2(Y, \nu))} <\infty$ and the convergence in \eqref{18vish2} will be understood in sense of the usual operator norm convergence. 
 \end{theorem}
 \begin{proof} Let $\mathfrak{t}_a$ be a $s$-nuclear, $0<s \leq 1$ operator. Then, it follows from Theorem \ref{DelTheorem2.8} that there exist two sequences $\{h_k\} \in L^p(\mathbb{Z}^n, L^2(Y, \nu))$ and $\{g_k\}_k \in L^{p'}(\mathbb{Z}^n, L^2(Y, \nu))$ with $$\sum_{k=1}^\infty \|h_k\|^s_{L^{p'}(\mathbb{Z}^n, L^2(Y, \nu))} \|g_k\|^s_{L^p(\mathbb{Z}^n, L^2(Y, \nu))} <\infty$$ such that 
 $$\mathfrak{t}_af(x)= \sum_{y \in \mathbb{Z}^n} \sum_{k} (h_k \otimes g_k)(x, y) f(y)\,\,\,\,\,\text{a.e.}\, x$$
 On the other hand by the definition of $\mathfrak{t}_a,$ we have
 $$\mathfrak{t}_af(x)= \int_{ \mathbb{T}^n} a(x, \xi)\, e^{2 \pi i x \cdot \xi} \mathcal{F}_{\mathbb{Z}^n}f(\xi)\, d\xi = \int \limits_{\mathbb{T}^n} a(x, \xi) e^{2 \pi i x \cdot \xi}\sum_{y \in \mathbb{Z}^n} e^{-2 \pi i y \cdot \xi} f(y) \, d\xi.$$
 Thus, we get that the equality 
 \begin{equation} \label{Vish39}
     \,\,\, \,\,\, \,\,\, \,\,\, \,\,\, \,\,\, \,\,\, \,\,\, \,\,\, \int \limits_{\mathbb{T}^n} a(x, \xi) \sum_{y \in \mathbb{Z}^n} e^{2 \pi i (x-y) \cdot \xi} f(y) \, d\xi = \sum_{y \in \mathbb{Z}^n} \sum_{k} (h_k \otimes g_k)(x, y) f(y),
 \end{equation} is true for every $f \in L^p(\mathbb{Z}^n, L^2(Y, \nu)).$ Now, for fix $v \in L^2(Y, \nu)$ and $l \in \mathbb{Z}^n,$ 
 consider the function $f_l \in L^p(\mathbb{Z}^n, L^2(Y, \nu))$ given by  
  $$f_l(y)= \begin{cases} v, & \text{if} \,\,y=l, \\ 0, & \text{otherwise}.\end{cases}$$
  For $f= f_l,$ equality \eqref{Vish39} in turn gives 
  $$ \mathcal{F}_{\mathbb{Z}^n}^{-1}[a(x, \cdot )v](x-l):= \int \limits_{\mathbb{T}^n} e^{2 \pi i (x-l)} a(x, \xi) v \, d\xi = \sum_{k} (h_k \otimes g_k)(x, l) v.$$
% Since $\lim_{M \rightarrow \infty}\sum_{k \leq M} (h_k \otimes g_k)(x, l)=K(x, y):=\sum_{k} (h_k \otimes g_k)(x, l)$  in operator norm we get that $\lim_{M \rightarrow \infty}\sum_{k \leq M} (h_k \otimes g_k)(x, l)v=K(x, y)v=\sum_{k} (h_k \otimes g_k)(x, l) v.$ So, 
 % $$ \mathcal{F}_{\mathbb{Z}^n}^{-1}[a(x, \cdot )v](x-l):= \int \limits_{\mathbb{T}^n} e^{2 \pi i (x-l)} a(x, \xi) v \, d\xi = \lim_{M \rightarrow \infty}\sum_{k \leq M} (h_k \otimes g_k)(x, l)v.$$
In view of this, by using the Fourier inversion formula for the discrete Fourier transform we get 
  \begin{align*}
      a(x, \xi)v & =\mathcal{F}_{\mathbb{Z}^n}[\mathcal{F}^{-1}_{\mathbb{Z}^n}a(x, \cdot)v](\xi)=  \sum_{l \in \mathbb{Z}^n} e^{-2 \pi i \xi \cdot l} \mathcal{F}^{-1}_{\mathbb{Z}^n}(a(x, \cdot)v)(l)  \\ &=  \sum_{l \in \mathbb{Z}^n} e^{-2 \pi i \xi \cdot (x-l)} \mathcal{F}^{-1}_{\mathbb{Z}^n}(a(x, \cdot)v)(x-l) \\&= \sum_{l \in \mathbb{Z}^n} e^{-2 \pi i \xi \cdot (x-l)} \sum_{k } (h_k \otimes g_k)(x, l)v \\ &= e^{-2 \pi i \xi \cdot x}  \sum_{k} h_k(x) \left( \sum_{l \in \mathbb{Z}^n} e^{2 \pi i \xi \cdot l} g_k(l) \right)v  \\& = e^{-2 \pi i \xi \cdot x} \sum_{k} h_k(x) \mathcal{F}_{\mathbb{Z}^n}g_k(-\xi) v \\ &= e^{-2 \pi i \xi \cdot x} \sum_{k} (h_k \otimes \mathcal{F}_{\mathbb{Z}^n} g_k)(x, -\xi)v \,\,\,\,\, \text{for every}\, v \in L^2(Y, \nu).
  \end{align*}
Therefore, 
         $$  a(x, \xi)= e^{-2 \pi i \xi \cdot x} \sum_{k} (h_k \otimes \mathcal{F}_{\mathbb{Z}^n}g_k)(x, -\xi),\,\,\, \text{a.e.}\,\, x \in \mathbb{Z}^n, \xi \in \mathbb{T}^n.$$
Conversely, assume that  there exist two sequences  $\{g_k\}_k \subset  L^p(\mathbb{Z}^n, L^2(Y, \nu))$ and  $\{h_k\}_k \subset L^{p'}(\mathbb{Z}^n, L^2(Y, \nu))$ with $\sum_{k=1}^\infty \|h_k\|^s_{L^{p'}(\mathbb{Z}^n, L^2(Y, \nu))} \|g_k\|^s_{L^p(\mathbb{Z}^n, L^2(Y, \nu))} <\infty$ such that  $$ a(x, \xi)= e^{-i2\pi x \cdot \xi} \sum_{k} (h_k \otimes \mathcal{F}_{\mathbb{Z}^n} g_k)(x, -\xi)\,\,\, x \in \mathbb{Z}^n, \xi \in \mathbb{T}^n.$$
     Now, since $a_N(x, \xi):= \lim_{N \rightarrow \infty} e^{-i2\pi x \cdot \xi} \sum_{k \leq N} (h_k \otimes \mathcal{F}_{\mathbb{Z}^n} g_k)(x, -\xi) $ converge to $a(x, \xi)$ in operator norm we have, for any $f \in L^p(\mathbb{Z}^n, L^2(Y, \nu)),$ 
     \begin{align*}
         \mathfrak{t}_a f(x) &= \int \limits_{\mathbb{T}^n} e^{2 \pi i  \xi \cdot x} a(x, \xi) \mathcal{F}_{\mathbb{Z}^n}f(\xi)\, d\xi= \int \limits_{\mathbb{T}^n} e^{2 \pi i  \xi \cdot x} \lim_{N \rightarrow \infty} a_N(x, \xi) \mathcal{F}_{\mathbb{Z}^n}f(\xi)\, d\xi \\ &= \int \limits_{\mathbb{T}^n} \left( \lim_{N \rightarrow \infty}\sum_{k \leq N} (h_k \otimes \mathcal{F}_{\mathbb{Z}^n}g_k)(x, -\xi) \right) \mathcal{F}_{\mathbb{Z}^n}f(\xi)\, d\xi \\ &=  \left( \lim_{N \rightarrow \infty}\sum_{k \leq N} \int \limits_{\mathbb{T}^n} (h_k \otimes \mathcal{F}_{\mathbb{Z}^n}g_k)(x, -\xi) \right) \mathcal{F}_{\mathbb{Z}^n}f(\xi)\, d\xi \\ &=   \lim_{N \rightarrow \infty}\sum_{k \leq N} \left(\int \limits_{\mathbb{T}^n} h_k(x) \lim_{M \rightarrow \infty} \sum_{|y| \leq M} e^{i2\pi y \cdot \xi } g_k(y) \right) \mathcal{F}_{\mathbb{Z}^n}f(\xi)\, d\xi\\ &=   \lim_{N \rightarrow \infty}\sum_{k \leq N} \left( h_k(x) \lim_{M \rightarrow \infty} \sum_{|y| \leq M}  g_k(y) \right) \left( \int \limits_{\mathbb{T}^n} e^{i2\pi y \cdot \xi } \mathcal{F}_{\mathbb{Z}^n}f(\xi)\, d\xi \right)\\ &=   \lim_{N \rightarrow \infty}\lim_{M \rightarrow \infty} \sum_{k \leq N} \sum_{|y| \leq M} (h_k \otimes g_k)(x,y) \left( \int \limits_{\mathbb{T}^n} e^{i2\pi y \cdot \xi } \mathcal{F}_{\mathbb{Z}^n}f(\xi)\, d\xi \right) \\ &=   \lim_{N \rightarrow \infty}\lim_{M \rightarrow \infty}  \sum_{|y| \leq M}\sum_{k \leq N} (h_k \otimes g_k)(x,y) f(y)  \\&= \sum_{y \in \mathbb{Z}^n} \left( \sum_{k} (h_k \otimes g_k)(x, y) \right) f(y).  
     \end{align*} 
Next, by applying Theorem\ref{DelTheorem2.8} we get that $\mathfrak{t}_a$ is a $s$-nuclear operator on the vector-valued Lebesgue space $L^p(\mathbb{Z}^n, L^2(Y, \nu)).$
     \end{proof}
Now, we present the following result. The proof is similar to the proof of Theorem \ref{decaying}.
 \begin{theorem}\label{decaying2}
  Let $1<p\leq 2,$     and let $\mathfrak{t}_\sigma$ be the vector-valued operator  associated to $\sigma(\cdot,\cdot).$ If $\mathfrak{t}_\sigma:L^{p}(\mathbb{Z}^n,L^2(Y,\nu))\rightarrow L^{p}(\mathbb{Z}^n,L^2(Y,\nu))$ is nuclear. Then $\Vert a(x,\xi)\Vert_{\textnormal{op}}\in L^{p'}(\mathbb{Z}^n\times \mathbb{T}^n),$ this means that
 \begin{equation}
     \Vert \Vert  a(x,\xi)\Vert_{\textnormal{op}}\Vert_{ L^{p'}(\mathbb{Z}^n\times \mathbb{T}^n)}:=\left( \int\limits_{\mathbb{T}^n}\sum\limits_{\xi\in\mathbb{Z}^n}\Vert a(x,\xi)\Vert^{p'}_{\textnormal{op}}dx\right)^{\frac{1}{p'}}<\infty.
 \end{equation} 
\end{theorem}

\subsection{Nuclear  pseudo-differential operators on $L^p(\mathbb{T}^n,H).$ Sufficient conditions.}

Let $H$ be a separable Hilbert space with infinite dimension. We fix an arbitrary orthonormal basis  $$B=\{e_\alpha:\alpha \in I\},$$  of $H.$ Let us recall that $H$ is separable if and only if  $|I|=\aleph_0.$ Our starting point is the following result which will be our main tool for the proof of our index formulae.

\begin{theorem}\label{TheoremNuclearSufficientCondition}
Let $0<s\leq 1$ and let $A:C^\infty(\mathbb{T}^n,H)\rightarrow C^\infty(\mathbb{T}^n,H)$ be a vector-valued periodic pseudo-differential operator with operator-valued symbol $\sigma:\mathbb{T}^n\times \mathbb{Z}^n\rightarrow \mathcal{L}(H).$ If   $1\leq p_i< \infty$, $i=1,2$ with the condition
\begin{equation}\label{Hiponsymbol}
    \Vert\{ \Vert \sigma(x,\xi)e_\alpha\Vert^s_{L^{p_2}(\mathbb{T}^n_x,H)}  \}_{(\xi,\alpha)} \Vert_{\ell^1(\mathbb{Z}^n\times I)}:= \sum_{\xi\in\mathbb{Z}^n}\sum_{\alpha\in I}\Vert \sigma(x,\xi)e_\alpha\Vert^s_{L^{p_2}(\mathbb{T}^n_x,H)}<\infty.
\end{equation}    The operator $A:L^{p_1}(\mathbb{T}^n,H)\rightarrow L^{p_2}(\mathbb{T}^n,H)$ extends to a $s$-nuclear operator. If $1\leq p<\infty$ and $p=p_1=p_2,$ in terms of the basis $B,$ the nuclear trace of $A$ is given by
\begin{equation}\label{nTraceformula1}
    \textnormal{\textbf{n-Tr}}(A)=\sum_{\xi\in\mathbb{Z}^n}\sum_{\alpha\in I}\int\limits_{\mathbb{T}^n}\langle \sigma(x,\xi)e_\alpha,e_\alpha\rangle_{H}dx. 
\end{equation} In particular, if $A$ is a Fourier multiplier with symbol $\sigma(\xi)$ then $\sigma(\xi)\in \mathfrak{S}_{s}(H),$ the Schatten-von Neumann class of order $s,$ and
\begin{equation}\label{nucleartracemultiplier}
    \textnormal{\textbf{n-Tr}}(A)=\sum_{\xi\in\mathbb{Z}^n}\textnormal{\textbf{n-Tr}}(\sigma(\xi)).
\end{equation}
\end{theorem}
\begin{proof} Let $f\in C^{\infty}(\mathbb{T}^n,H).$ For $\xi\in\mathbb{Z}^n,$ the Fourier coefficient $\widehat{f}(\xi)$ can be written as,
\begin{equation}
    \widehat{f}(\xi)=\sum_{\alpha\in I}\widehat{f}(\alpha,\xi)e_\alpha,\,\,\,\widehat{f}(\alpha,\xi):=\langle\widehat{f}(\xi),e_\alpha \rangle_H=\int\limits_{\mathbb{T}^n}e^{-i2\pi x\cdot \xi}\langle \widehat{f}(\xi),e_\alpha \rangle_{H}.
\end{equation} In terms of the basis $B,$ the operator $A$ has the form,
\begin{equation*}
    Af(x)=\sum_{\xi\in\mathbb{Z}^n}e^{i2\pi x\cdot \xi}\sigma(x,\xi)\widehat{f}(\xi)=\sum_{\xi\in\mathbb{Z}^n}\sum_{\alpha\in I}e^{i2\pi x\cdot \xi}\widehat{f}(\alpha,\xi)\sigma(x,\xi)e_\alpha.
\end{equation*} In fact, the fact that $\sigma(x,\xi)\in \mathcal{L}(H)$ for all $(x,\xi)\in \mathbb{T}^n\times \mathbb{Z}^n,$ implies that
\begin{eqnarray*}
\sigma(x,\xi)\widehat{f}(\xi)=\sum_{\alpha\in I}\widehat{f}(\alpha,\xi)\sigma(x,\xi)e_\alpha,
\end{eqnarray*} where the convergence is understood in the topology induced by the norm of $H$. Now, we will prove that the hypothesis \eqref{Hiponsymbol} does of the following sequences $\{g_{\alpha,\xi}\}\subset L^{p_1'}(\mathbb{T}^n,H)$ and $\{h_{\alpha,\xi}\}\subset L^{p_2}(\mathbb{T}^n,H)$ where
\begin{equation*}
    \langle f, g_{\alpha,\xi}\rangle:=\widehat{f}(\alpha,\xi),\,\,h_{\alpha,\xi}(x):=e^{i2\pi x\cdot \xi}\sigma(x,\xi)e_\alpha \in H,
\end{equation*}  a nuclear decomposition of $A.$ In fact, it will be suffice to prove that
$$ 
\sum_{\xi\in\mathbb{Z}^n}\sum_{\alpha\in I} \Vert g_{\alpha,\xi}\Vert^s_{ L^{p_1'}(\mathbb{T}^n,H)   } \Vert h_{\alpha,\xi}\Vert^s_{L^{p_2}(\mathbb{T}^n_x,H)}<\infty.$$ To do so, let us observe that
\begin{align*}
     \Vert g_{\alpha,\xi}\Vert_{ L^{p_1'}(\mathbb{T}^n,H)} &=\sup_{\Vert f\Vert_{  L^{p_1}(\mathbb{T}^n,H)=1  }} |\langle f, g_{\alpha,\xi}\rangle| \\
     &=\sup_{\Vert f\Vert_{  L^{p_1}(\mathbb{T}^n,H)=1  }} |\langle \widehat{f}(\xi), e_\alpha\rangle_{H}| \\
     &\leq \sup_{\Vert f\Vert_{  L^{p_1}(\mathbb{T}^n,H)=1  }} \Vert  \widehat{f}(\xi)\Vert_{H}\Vert e_\alpha\Vert_{H}\\&=\sup_{\Vert f\Vert_{  L^{p_1}(\mathbb{T}^n,H)=1  }} \Vert  \widehat{f}(\xi)\Vert_{H},
\end{align*} where we have used that $\left\Vert e_\alpha\right\Vert_{H}=1$ for all $ \alpha.$ The estimate 
$$\Vert\widehat{f}(\xi)\Vert_{H}=\left\Vert\int\limits_{\mathbb{T}^n}e^{-i2\pi x\cdot \xi}f(x)dx\right\Vert_{H}\leq\int\limits_{\mathbb{T}^n}\left\Vert f(x)\right\Vert_{H}dx\leq \Vert f \Vert_{  L^{p_1}(\mathbb{T}^n,H)}, $$ implies that $\Vert g_{\alpha,\xi}\Vert_{ L^{p_1'}(\mathbb{T}^n,H)}\leq 1.$ So, we have,
\begin{align*}
    \sum_{\xi\in\mathbb{Z}^n}\sum_{\alpha\in I} \Vert g_{\alpha,\xi}\Vert^s_{ L^{p_1'}(\mathbb{T}^n,H)   }  &\Vert h_{\alpha,\xi}\Vert^s_{L^{p_2}(\mathbb{T}^n,H)}\\
    &\leq \sum_{\xi\in\mathbb{Z}^n}\sum_{\alpha\in I}  \Vert h_{\alpha,\xi}\Vert^s_{L^{p_2}(\mathbb{T}^n,H)}\\
    &= \sum_{\xi\in\mathbb{Z}^n}\sum_{\alpha\in I}  \Vert \sigma(x,\xi)e_\alpha\Vert^s_{L^{p_2}(\mathbb{T}^n_x,H)}<\infty\,\,\,\,\,\,.
\end{align*} Therefore, we have proved the $s$-nuclearity of the operator $A:L^{p_1}(\mathbb{T}^n,H)\rightarrow L^{p_2}(\mathbb{T}^n,H)$. Now, if $p=p_1=p_2,$ the nuclear trace of $A$ is given by
\begin{align*}
   \,\,\,\,\,\,\,\,\,\,\,\,\,\,\,\,\,\,\,\,\,\,\,\,\,\,\,\,\,\,\,\,\,\,\,\,\,\,\,\,\,\,\,\,\,\, \textnormal{\textbf{n-Tr}}(A) &=\sum_{\xi\in\mathbb{Z}^n}\sum_{\alpha\in I}\langle h_{\alpha,\xi},g_{\alpha,\xi} \rangle\\ &=\sum_{\xi\in\mathbb{Z}^n}\sum_{\alpha\in I}\langle\widehat{h}_{\alpha,\xi}(\xi),e_\alpha \rangle_{H}\\&=\sum_{\xi\in\mathbb{Z}^n}\sum_{\alpha\in I}\langle \int\limits_{\mathbb{T}^n}h_{\alpha,\xi}(x)e^{-i2\pi x\cdot \xi}dx,e_\alpha \rangle_{H}\\
     &=\sum_{\xi\in\mathbb{Z}^n}\sum_{\alpha\in I}\langle \int\limits_{\mathbb{T}^n}e^{i2\pi x\cdot \xi}\sigma(x,\xi)e^{-i2\pi x\cdot \xi}dx,e_\alpha \rangle_{H}\\
     &=\sum_{\xi\in\mathbb{Z}^n}\sum_{\alpha\in I} \int\limits_{\mathbb{T}^n}\langle\sigma(x,\xi),e_\alpha \rangle_{H}dx.
\end{align*} Now, we assume that $A$ is a Fourier multiplier. Since, $\sigma(\xi)\in \mathcal{L}(H)$ for all $\xi\in  \mathbb{Z}^n,$ we have
\begin{eqnarray}
\sigma(\xi)v=\sum_{\alpha\in I}\langle v,e_\alpha \rangle_H \sigma(\xi)e_\alpha.
\end{eqnarray} We claim that $\sigma(\xi)\in \mathfrak{S}_{s}(H)=\mathfrak{N}_s(H,H).$ In fact, we only need to prove that the sequences $\{e_{\alpha}'\}\subset H'$ and $\{h_{\alpha,\xi}\}\subset H$ provide a nuclear decomposition for $\sigma(\xi),$ where $e_{\alpha}'(v):=\langle v,e_\alpha \rangle_H$ and $h_{\alpha,\xi}:=\sigma(\xi)e_\alpha.$ For this, it will be suffice to prove that $$ \sum_{\alpha\in I}\Vert e_\alpha'\Vert_{H}^s\Vert h_{\alpha,\xi}\Vert_{H'}^s=\sum_{\alpha\in I}\Vert h_{\alpha,\xi}\Vert_{H}^s =\sum_{\alpha\in I}\Vert \sigma(\xi)e_\alpha\Vert_{H}^s <\infty, $$
where we have used that  $\Vert e_\alpha'\Vert_{H'}=1, $ for all $\alpha\in I.$ However, we can do this by taking into account that for Fourier multipliers, the condition \eqref{Hiponsymbol} implies,
\begin{equation*}
 \sum_{\alpha\in I}\Vert \sigma(\xi)e_\alpha\Vert_{H}^s\leq   \sum_{\xi\in\mathbb{T}^n}\sum_{\alpha\in I}\Vert \sigma(\xi)e_\alpha\Vert^s_{(H)}=\sum_{\xi\in\mathbb{T}^n}\sum_{\alpha\in I}\Vert \sigma(x,\xi)e_\alpha\Vert^s_{L^{p_2}(\mathbb{T}^n_x,H)}<\infty.
\end{equation*} By observing that the nuclear trace of $\sigma(\xi)$ is given by
\begin{align*}
    \textnormal{\textbf{n-Tr}}(\sigma(\xi))=\sum_{\alpha\in I}\langle h_\alpha,e_{\alpha}'\rangle=\sum_{\alpha\in I}\langle \sigma(\xi)e_{\alpha},e_\alpha\rangle_H,\,\,\,\,\,\,\,\,\,\,\,\,\,\,\,\,\,\,\,\,\,\,\,\,\,
\end{align*} from \eqref{nTraceformula1} we conclude that,
\begin{align*}
   \textnormal{\textbf{n-Tr}}(A)&= \sum_{\xi\in\mathbb{Z}^n}\sum_{\alpha\in I} \int\limits_{\mathbb{T}^n}\langle\sigma(\xi),e_\alpha \rangle_{H}dx \\ &= \sum_{\xi\in\mathbb{Z}^n}\sum_{\alpha\in I} \langle\sigma(\xi),e_\alpha \rangle_{H}= \sum_{\xi\in\mathbb{Z}^n}\textnormal{\textbf{n-Tr}}(\sigma(\xi)).
\end{align*} Thus, we finish the proof.

\end{proof}

Let us observe that \eqref{nucleartracemultiplier} suggests that one can extend the trace formula \eqref{nTraceformula1} to a similar expression where does not appear the basis $H.$ This will be useful because we could show, under certain conditions, that  the nuclear trace of a vector-valued pseudo-differential operator depends only on the nuclear trace of the values of its symbol.

\section{Applications to Index theory}\label{indexsection}
\subsection{The index of vector-valued Fourier multipliers}

In this subsection, we show that how to use the  nuclearity results developed in Section \ref{nuche} to the index theory of vector-valued Fourier multipliers.
We are interested in computing the index of Fredholm periodic vector-valued pseudo-differential operators. The following definition will be useful for our further applications.

\begin{definition}\label{mBelliptic}
Let $H$ be a  separable complex  Hilbert space. Let us assume that $m>0$  and let $B=\{e_\alpha:\alpha\in I\}$ be an orthonormal basis of $H.$ We say that a Fourier multiplier  $A:L^2_{\textnormal{fin} \,B}(\mathbb{T}^n,H)\rightarrow L^2(\mathbb{T}^n,H)$ is $(m,B)$-elliptic if, 
\begin{itemize}
    \item the operator-valued symbol $\sigma:=\{\sigma(\xi)\}_{\xi\in\mathbb{Z}^n}$ of $A$ consists of a sequence of densely defined operators  $\sigma(\xi),$ $\xi\in\mathbb{Z}^n$, such that $B\subset \cap_{\xi\in\mathbb{Z}^n}\textnormal{Dom}(\sigma(\xi)),$ and 
    \item  
    \begin{equation}\label{lowerboundeigenvalue}
  \sup_{0<t<\infty} e^{t C (1+|\alpha|+|\xi|)^{m}}\Vert {e^{-t \sigma(\xi)^* \sigma(\xi)}}e_\alpha \Vert_{H},\sup_{0<t<\infty} e^{t C(1+|\alpha|+|\xi|)^{m}}\Vert {e^{-t \sigma(\xi) \sigma(\xi)^*}}e_\alpha \Vert_{H}<\infty,
\end{equation}   for some positive constant $C>0.$
\end{itemize}
\end{definition}
\begin{remark}
 The existence of $(m,B)$-elliptic Fourier multipliers will be considered in the next subsection. Indeed, we will prove that this type of operators arise naturally in the context of  periodic H\"ormander classes.
\end{remark}

 Now, we  assume that $H$ is an infinite-dimensional  separable Hilbert space with a basis  $B=\{e_\alpha:\alpha\in \mathbb{Z}^\kappa\}$  indexed by the discrete set $I=\mathbb{Z}^\kappa$, for some integer $\kappa\in\mathbb{N}.$
The following lemma summarises the nuclear properties of the heat semigroup defined by the Laplacians   $\Delta_a=A^*A$ and $\Delta_b=AA^*$ associated to $A.$
\begin{lemma}
    Let $t>0.$ Let $A$ be a $(m,B)$-elliptic Fourier multiplier. Then, the operators $e^{-t\Delta_a}$ and $e^{-t\Delta_b}$ extend to a $s$-nuclear operators from $L^{p_1}(\mathbb{T}^n,H)$ into $L^{p_2}(\mathbb{T}^n,H),$ for all $0<s\leq 1$ and $1\leq p_i<\infty.$ In this case, their nuclear traces are given by
    \begin{equation}\label{Formulaheatkerneltrace}
        \textnormal{\textbf{n-Tr}}(e^{-t\Delta_a})=\sum_{\xi\in\mathbb{Z}^n} \textnormal{\textbf{n-Tr}}(e^{-t \sigma(\xi)^* \sigma(\xi) }) \textnormal{   and,   } \textnormal{\textbf{n-Tr}}(e^{-t\Delta_b})=\sum_{\xi\in\mathbb{Z}^n} \textnormal{\textbf{n-Tr}}(e^{-t \sigma(\xi) \sigma(\xi)^* }).
    \end{equation}
\end{lemma}

\begin{proof}
Let $t>0.$ The operator-valued symbol of $e^{-t\Delta_a}$ is given by  $\sigma_{e^{-t\Delta_a}}(\xi)=e^{-t \sigma(\xi)^* \sigma(\xi) }.$ If we prove that
\begin{equation*}
    \sum_{\xi\in\mathbb{Z}^n}\sum_{\alpha\in I}\Vert \sigma_{e^{-t\Delta_a}}(\xi)e_\alpha \Vert^s_{H}=\sum_{\xi\in\mathbb{Z}^n}\sum_{\alpha\in I}\Vert {e^{-t \sigma(\xi)^* \sigma(\xi)}}e_\alpha \Vert^s_{H}<\infty,
\end{equation*} then, from Theorem \ref{TheoremNuclearSufficientCondition}, we can deduce the $s$-nuclearity of $e^{-t\Delta_a}$ as well as the formula \eqref{Formulaheatkerneltrace}. From the $(m,B)$-ellipticity of $A$ we have
\begin{equation*}
    \Vert {e^{-t \sigma(\xi)^* \sigma(\xi)}}e_\alpha \Vert^s_{H}\leq C'e^{-Ct s(1+|\alpha|+|\xi|)^{m}}.
\end{equation*} We end the proof by summing over $(\xi,\alpha)\in \mathbb{Z}^{n+\kappa}$ and taking into account that $m>0.$ Indeed, we have
\begin{equation*}
    \sum_{\xi\in\mathbb{Z}^n}\sum_{\alpha\in I}\Vert \sigma_{e^{-t\Delta_a}}(\xi)e_\alpha \Vert^s_{H}\leq \sum_{\xi\in\mathbb{Z}^n}\sum_{\alpha\in I} Ce^{-t s(1+|\alpha|+|\xi|)^{m}} <\infty.
\end{equation*} We can use an analogous argument for $e^{-t\Delta_b}.$  This completes the proof. 
\end{proof}

The following lemma presents  a necessary condition for the Fredholmness of a Fourier multiplier.
\begin{lemma}\label{nece}
Let us consider a closed and densely defined vector-valued Fourier multiplier $A$ on $L^{2}(\mathbb{T}^n,H)$ with operator-valued symbol $\sigma.$ If $A$ is Fredholm, then for every $\eta\in\mathbb{Z}^n,$ $\sigma(\eta)$ extends to a Fredholm operator on $H.$
\end{lemma}
\begin{proof}
We  need to show that if $A$ is Fredholm, then for all $\eta\in\mathbb{Z}^n,$ $\sigma(\eta)$ has kernel and cokernel of finite dimension. Since the symbol  of the adjoint $A^{*}$ of $A$ is $\sigma^*\{\sigma(\eta)^*\}_{\eta\in\mathbb{Z}^n},$
and $A$ is Fredholm if only if $A^*$ is Fredholm, it is sufficient to show that $$\dim \textnormal{Ker}(\sigma(\eta))<\infty,\,\,\,\forall \eta\in\mathbb{Z}^n.$$
By \eqref{symbol}, we deduce that for all $v\in H,$ $\eta\in\mathbb{Z}^n,$
\begin{equation}
A(e^{i2\pi x\cdot\eta}v)=e^{i2\pi x\cdot\eta}\sigma(\eta)v,\,\,\,v\in H.
\end{equation}
Hence if $\{v_{i,\eta}\}$ is a basis for $\textnormal{Ker}(\sigma(\eta)),$ then $\{\phi_{i,\eta}:=x\mapsto v_{i,\eta}e^{ix\cdot \eta}\}$ is a set of linearly independent functions in $\textnormal{Ker}(A),$ which has finite dimension. So, $\textnormal{Ker}(\sigma(\eta))$ is finite-dimensional. Thus, we finish the proof.
\end{proof}
The following tool is known as the McKean-Singer theorem. We present a proof for completeness.
\begin{lemma}\label{heatapproach}
Let us assume that $T:H_{1}\rightarrow H_{2}$ is a Fredholm operator, $TT^*$ and $T^*T$ have discrete spectrum, and for all $t>0,$\,  $ e^{-tT^*T}$ and $e^{-tTT^*}$ are trace class. Then \begin{equation}
\textnormal{\textbf{Ind}}(T)=\textnormal{\textbf{Tr}}(e^{-tT^*T})-\textnormal{\textbf{Tr}}(e^{-tTT^*}).
\end{equation}  
\end{lemma}
\begin{proof}
Let $\lambda \in \textnormal{Spect}(T^*T),$ $\lambda\neq 0,$ then there exists a non-zero vector $\phi\in H_{1}$ such that $T^*T(\phi)= \lambda \phi$ and therefore we have $TT^*T(\phi)= \lambda T\phi.$ Since $T(\phi)$ is non-zero, it is an eigenvalue of $TT^*.$ It follows that the non-zero eigenvalues of $T^*T$ and $TT^*$ are the same. Thus
$$ \textnormal{\textbf{Tr}}(e^{-tT^*T})-\textnormal{\textbf{Tr}}(e^{-tTT^*})=\dim \text{Ker} (T^{*}T)-\dim \text{Ker} (TT^{*}).$$
Since $\text{Ker}(T^*T)=\text{Ker}(T)$ and $\text{Ker}(TT^*)=\text{Ker}(T^*)$ the proof follows.
\end{proof}

Now,  with the machinery developed above, we present the main theorem  of this subsection.

\begin{theorem}\label{Indextheorem1} 
 Let us assume that the Fourier multiplier $A\equiv T_\sigma:L^2(\mathbb{T}^n,H)\rightarrow L^2(\mathbb{T}^n,H)$ is a (closed and densely defined) Fredholm operator. If $A$ is $(m,B)$-elliptic then the analytical index of $A$  is given by 
\begin{equation}\label{IndexFormulaNov2018}
    \textnormal{\textbf{Ind}}(A)=\sum_{\eta\in \mathbb{Z}^n}\textnormal{\textbf{Ind}}(\sigma(\eta)).
\end{equation}
\end{theorem}
\begin{proof}
    By Lemma \ref{elpapadeloslemma}, we have
\begin{equation}\label{discrete}
\textnormal{Spect}(A)=\bigcup_{\eta\in\mathbb{Z}^n}\textnormal{Spect}(\sigma(\eta)).
\end{equation} So, by using the fact that $A$ is $(m,B)$-elliptic we deduce that $\textnormal{Spect}(\sigma(\eta))$ is a discrete set for all $\eta$ and we conclude  that   $\textnormal{Spect}(A)$ is also a discrete set. 
By applying Lemma \ref{heatapproach} and Lemma \ref{nece} we deduce that every $\sigma(\eta)$ is Fredholm for all $\eta,$ and consequently,
\begin{equation}
\textnormal{\textbf{Ind}}[\sigma(\eta)]=\textnormal{\textbf{sp-Tr}}[e^{-t\sigma(\eta)^*\sigma(\eta)}]-\textnormal{\textbf{sp-Tr}}[e^{-t\sigma(\eta)\sigma(\eta)^*}].
\end{equation}
 By using, Lemma \ref{heatapproach},  Lemma \ref{nece}, the Grothendieck-Lidskii theorem for $s=\frac{2}{3},$ and  \eqref{Formulaheatkerneltrace}, we obtain
\begin{align*}
\textnormal{\textbf{Ind}}[A]&=\textnormal{\textbf{sp-Tr}}[e^{-tA^*A}]-\textnormal{\textbf{sp-Tr}}[e^{-tAA^*}]=\textnormal{\textbf{n-Tr}}[e^{-tA^*A}]-\textnormal{\textbf{n-Tr}}[e^{-tAA^*}]\,\,\,\,\,\,\,\\
&=\sum_{\eta\in\mathbb{Z}^n}(\textnormal{\textbf{n-Tr}}[e^{-t\sigma(\eta)^*\sigma(\eta)}]-\textnormal{\textbf{n-Tr}}[e^{-t\sigma(\eta)\sigma(\eta)*}])\\
&=\sum_{\eta\in\mathbb{Z}^n}(\textnormal{\textbf{sp-Tr}}[e^{-t\sigma(\eta)^*\sigma(\eta)}]-\textnormal{\textbf{sp-Tr}}[e^{-t\sigma(\eta)\sigma(\eta)^*}])\\
&=\sum_{\eta\in\mathbb{Z}^n}\textnormal{\textbf{Ind}}[\sigma(\eta)],
\end{align*}
which concludes the proof of the theorem. 
\end{proof}

\subsection{An index formula for elliptic  operators on periodic H\"ormander classes}
 In this section we will use our index formula \eqref{IndexFormulaNov2018} in order to study the index of periodic elliptic pseudo-differential operators $T$ defined on $L^2(\mathbb{T}^n\times\mathbb{T}^m)$ where $m$ and $n$ are integer numbers. By following Ruzhansky and Turunen \cite{RT1}, if $T$ is a continuous operator on $C^\infty( \mathbb{T}^\kappa),$  there exists a function $\sigma_T:\mathbb{T}^\kappa\times \mathbb{Z}^\kappa\rightarrow\mathbb{C} ,$ satisfying
 \begin{eqnarray}
 Tf(x)=\sum_{\xi\in\mathbb{Z}^\kappa}e^{i2\pi x\cdot \xi}\sigma_T(x,\xi)(\mathscr{F}_{\mathbb{T}^\kappa}{f})(\xi),
 \end{eqnarray}
for every $f\in C^\infty(\mathbb{T}^\kappa).$ The function $\sigma_T$ is the so called, full symbol, associated to $T.$ An important property in this pseudo-differential calculus is the relation
\begin{eqnarray}
\sigma_T(x,\xi)=e_{-\xi}(x)T(e_\xi)(x),\,\,e_\xi(x):=e^{i2\pi x\cdot \xi},\xi\in\mathbb{Z}^\kappa,
\end{eqnarray} which shows the uniqueness of the symbol. In particular, if $\kappa=n+m$ and $T$ is a continuous operator on $C^\infty(\mathbb{T}^n\times \mathbb{T}^m)$ we will write
\begin{eqnarray}
Tf(x,y)=\sum_{\xi\in\mathbb{Z}^n}\sum_{\eta\in\mathbb{Z}^m}e^{i2\pi(x,y)\cdot (\xi,\eta)}a(x,y,\xi,\eta)(\mathscr{F}_{\mathbb{T}^n\times \mathbb{T}^m }{f})(\xi,\eta)
\end{eqnarray} for every $f\in C^\infty( \mathbb{T}^n\times \mathbb{T}^m).$ The symbol $\sigma_T((x,y),(\xi,\eta))=a(x,y,\xi,\eta)$ is a function defined on $\mathbb{T}^n\times \mathbb{T}^m\times \mathbb{Z}^n\times \mathbb{Z}^m.$ An important question arose in the previous section is the existence of Fredholm vector-valued Fourier multipliers  on $L^2(\mathbb{T}^n,H)$ satisfying the condition of $(m,B)$-ellipticity. In order to answer this question, we will associate every elliptic operator $T$ on $C^\infty(\mathbb{T}^n\times \mathbb{T}^m)$  to a vector-valued Fourier multiplier $A_T$ on $C^\infty(\mathbb{T}^n,L^2(\mathbb{T}^n))$ satisfying
\begin{align}
    \textnormal{\textbf{Ind}}(T)=\textnormal{\textbf{Ind}}(A_T).
\end{align} In order to present the construction of the vector-valued Fourier multiplier $A_T,$ we need the following remark.

\begin{remark}[Construction of $A_T$] Let $T$ be a continuous elliptic  operator on $C^\infty(\mathbb{T}^n\times \mathbb{T}^m).$ If $a(x,y,\xi,\eta)$ denotes the symbol of $T$ then for every $(x,\xi)\in\mathbb{T}^n\times \mathbb{Z}^n,$ consider the continuous operator $\sigma{(x,\xi)}:C^\infty(\mathbb{T}^m) \rightarrow C^\infty(\mathbb{T}^m),$ defined by
\begin{eqnarray}
\sigma{(x,\xi)}g(y)=\sum_{\eta\in\mathbb{Z}^n}e^{i 2\pi y\cdot \eta}a(x,y,\xi,\eta)(\mathscr{F}_{\mathbb{T}^m}g)(\eta),\,\,\,g\in C^\infty(\mathbb{T}^m).
\end{eqnarray}
 Clearly, $\sigma{(x,\xi)}$ is a periodic pseudo-differential operator with symbol $a(x,\cdot, \xi, \cdot).$ Now, let us define $A_T$ as the vector-valued pseudo-differential operator associated to the symbol  $$\sigma:=\{\sigma(x,\xi)\}_{(x,\xi)\in \mathbb{T}^n\times \mathbb{Z}^n}.$$ In terms of the orthonormal basis $B=\{e_\eta:\eta\in\mathbb{Z}^n\},$ where $e_{\eta}(y):=e^{i2\pi y\cdot \eta},$ of $L^2(\mathbb{T}^m)$
  the operator $A_T:L^2_{fin\, B}(\mathbb{T}^n,L^2(\mathbb{T}^m))\rightarrow L^2(\mathbb{T}^n,L^2(\mathbb{T}^m)) $ is defined by the formula
 \begin{eqnarray}
 (A_T u)(x)=\sum_{\xi\in \mathbb{Z}^n}e^{i2\pi x\cdot \xi}\sigma(x,\xi)\widehat{u}(\xi),
 \end{eqnarray}where $u\in L^2_{fin\, B}(\mathbb{T}^n,L^2(\mathbb{T}^m))$ and $\widehat{u}(\xi)\in L^2(\mathbb{T}^m)$ is the vector-valued Fourier transform of $u$ at $\xi\in\mathbb{Z}^n.$
\end{remark}

Now,  we restrict our attention to periodic pseudo-differential operator of the form
\begin{equation}\label{spplit}
    Tf(x,y)=\sum_{(\xi,\eta)\in \mathbb{Z}^n\times \mathbb{Z}^m}e^{i2\pi(x,y)\cdot (\xi,\eta)}a(y,\xi,\eta)\widehat{f}(\xi,\eta),\,\,x\in \mathbb{T}^n,\,y\in\mathbb{T}^m.
\end{equation} With this restriction the vector-valued symbol associated to $A_T$ is independent of the variables $x,$ and consequently $A_T$ is a vector-valued Fourier multiplier. In this framework, we can apply the vector-valued index theorem of the previous section.

We assume that $T$ is an elliptic operator on $\mathbb{T}^{n}\times \mathbb{T}^{m}$ and that $T$ belongs to the H\"ormander class $\Psi^\varkappa_{\rho,\delta}( \mathbb{T}^{m}\times \mathbb{Z}^{n}\times \mathbb{Z}^{m}),$ $0\leq \delta<\rho\leq 1,$ for some positive order $\varkappa>0.$ This means that we have estimates of the form
\begin{equation}
    |\partial_y^{\gamma}\Delta_{\xi}^{\alpha}\Delta_\eta^{\omega}\sigma(y,\xi,\eta)|\leq C_{\beta,\gamma,\alpha,\omega}\langle \xi,\eta\rangle^{\varkappa-\rho(|\alpha|+|\omega|)+\delta|\gamma| },\,\,\beta,\gamma\in \mathbb{N}_0^n,\,\,\alpha,\omega\in \mathbb{N}_0^m,
\end{equation}
and 
\begin{equation}
    |\sigma(y,\xi,\eta)|\geq C_\sigma\langle \xi,\eta\rangle^{\varkappa},\,\,|\xi|+|\eta|\geq M, \textnormal{ for some }M>0,
\end{equation} where $\langle \xi,\eta\rangle:=(1+|\eta|^2+|\xi|^2)^{\frac{1}{2}}.$ These are sufficient and necessary conditions for the Fredholmness of the operator $A$ on $L^2(\mathbb{T}^{n}\times \mathbb{T}^{m}),$ (see  Ruzhansky and Turunen \cite{RT1,Ruz}). 

Let us note that we can identify $C^\infty(\mathbb{T}^{n}\times \mathbb{T}^{m})\cong C^\infty(\mathbb{T}^{n}, C^\infty(\mathbb{T}^{m})) $ by using the isomorphism $$ f\rightsquigarrow \alpha_f:C^\infty(\mathbb{T}^{n}\times \mathbb{T}^{m})\rightarrow  C^\infty(\mathbb{T}^{n}, C^\infty(\mathbb{T}^{m})),\,\,\,\alpha_f(x)(y)=f(x,y),    $$ and moreover, we also have the identification $L^2(\mathbb{T}^{n}\times \mathbb{T}^{m})\cong L^2(\mathbb{T}^{n}, L^2(\mathbb{T}^{m})).$ In this sense, the periodic pseudo-differential operator $T$, can be realised as a vector-valued Fourier multiplier $A_T.$ In fact,

\begin{equation}\label{spplit22}
Af(x,\cdot)\equiv    A_T(\alpha_f)(x)=\sum_{\xi\in \mathbb{Z}^n}e^{i2\pi x\cdot \xi}a(\xi)\widehat{\alpha}_f(\xi)\in L^2(\mathbb{T}^m),\,\,x\in \mathbb{T}^n,\,\,\,f\in C^\infty(\mathbb{T}^m),
\end{equation} where the  operator-valued symbol $a:\mathbb{Z}^n\rightarrow \mathcal{L}(H^s(\mathbb{T}^n),  H^{s-\varkappa}(\mathbb{T}^n) ),$ is the periodic pseudo-differential operator defined by \begin{equation}\label{eureka}
    a(\xi)g(y)\equiv \textnormal{Op}(\sigma(\cdot,\xi,\cdot))g(y):=\sum_{\eta\in \mathbb{Z}^m }e^{i 2\pi x\cdot \xi}\sigma(y,\xi,\eta)\widehat{g}(\eta),\,\,g\in C^{\infty}(\mathbb{T}^m).
\end{equation}
With the notation above, we present the following index formula. 
By properties of elliptic pseudo-differential operators, the kernel and the cokernel of elliptic operators consist of smooth functions. With this property in mind, we see that $\dim \textnormal{Ker}(T)=\dim \textnormal{Ker}(A_T)$ and $\dim \textnormal{Coker}(T)=\dim \textnormal{Coker}(A_T).$ Therefore, we deduce that $\textnormal{\textbf{Ind}}(T)= \textnormal{\textbf{Ind}}(A_T).$

\begin{theorem} \label{periodicindex2018} Let us assume that $T$ as in \eqref{spplit} is an elliptic operator on $\mathbb{T}^{n}\times \mathbb{T}^{m}$ and that $T$ belongs to the H\"ormander class $\Psi^\varkappa_{\rho,\delta}( \mathbb{T}^{m}\times \mathbb{Z}^{n}\times \mathbb{Z}^{m}),$ $0\leq \delta<\rho\leq 1,$ for some $\varkappa>0.$ Assume that $A_T$ is  $(\varkappa,B)$-elliptic, where  $B=\{e_{\eta}:\eta\in\mathbb{Z}^m \},$ $e_\eta(y)=e^{i2\pi \eta y},$ $y\in \mathbb{T}^m.$  Then, the index of $T$ is given by
\begin{equation}\label{indexrealised}
    \textnormal{\textbf{Ind}}(T)=\sum_{\xi\in\mathbb{Z}^n}\textnormal{\textbf{Ind}}(a(\xi)),
\end{equation} where the operator-valued symbol $a=(a(\xi))_{\xi\in\mathbb{Z}^n}$ is defined as in \eqref{eureka}. \end{theorem}

\begin{proof}
   The Fredholmness of $T$ is a consequence of the ellipticity condition on its symbol $a(\cdot,\cdot,\cdot).$ Since we know that  $\dim \textnormal{Ker}(T)=\dim \textnormal{Ker}(A_T)$ and $\dim \textnormal{Coker}(T)=\dim\textnormal{Coker}(A_T),$ we obtain the Fredholmness of $A_T$ on $L^2(\mathbb{T}^n,L^2(\mathbb{T}^m)).$ So, we deduce that $\textnormal{\textbf{Ind}}(T)= \textnormal{\textbf{Ind}}(A_T).$ In order to  use Theorem \ref{Indextheorem1} to deduce \eqref{indexrealised}, we only need to use that $A_T$ is $(\varkappa,B)$-elliptic, where $B,$ in this case, is the orthonormal basis of $L^2(\mathbb{T}^m)$ given by  $B=\{e_{\eta}:\eta\in\mathbb{Z}^m \},$ $e_\eta(y)=e^{i2\pi \eta y},$ $y\in \mathbb{T}^m.$ So, we finish the proof.
\end{proof}
\begin{example} In the previous theorem we have assumed that the operator $A_T$ is $(\varkappa,B)$-elliptic. The main goal of this example is to show that this condition arises naturally, in the context of periodic operators,  for Fourier multipliers on $C^\infty(\mathbb{T}^n\times \mathbb{T}^m ).$
Let us consider the Fourier multiplier
\begin{equation}
    Tf(x,y)=\sum_{\xi\in\mathbb{Z}^n}\sum_{\eta\in\mathbb{Z}^m}e^{i2\pi(x,y)\cdot (\xi,\eta)}a(\xi,\eta)(\mathscr{F}_{\mathbb{T}^n\times \mathbb{T}^m }{f})(\xi,\eta),\,\,x\in\mathbb{T}^n,\,y\in\mathbb{T}^m,
\end{equation} where the symbol $a$ is elliptic and of order $\varkappa>0.$ This means that
\begin{equation}
    |\Delta_{\xi}^{\alpha}\Delta_\eta^{\omega}\sigma(\xi,\eta)|\leq C_{\alpha,\omega}\langle \xi,\eta\rangle^{\varkappa-\rho(|\alpha|+|\omega|) },\,\,\,\,\alpha,\omega\in \mathbb{N}_0^m,
\end{equation}
and 
\begin{equation}\label{lowebounellip}
    |\sigma(\xi,\eta)|\geq C_\sigma\langle \xi,\eta\rangle^{\varkappa},\,\,|\xi|+|\eta|\geq M, \textnormal{ for some }M>0.
\end{equation}
We want to illustrate that $T$ satisfies the hypothesis of Theorem \ref{periodicindex2018}. So, we only need to show that $A_T$ is $(\varkappa,B)$-elliptic, where  $B=\{e_{\eta}:\eta\in\mathbb{Z}^m \},$ $e_\eta(y)=e^{i2\pi \eta y},$ $y\in \mathbb{T}^m.$ So, we will prove that
$$\sup_{0<t<\infty} e^{tC (1+|\eta|+|\xi|)^{m}}\Vert {e^{-t a(\xi)^* a(\xi)}}e_\eta \Vert_{L^2(\mathbb{T}^m)}$$ and $$ \sup_{0<t<\infty} e^{tC (1+|\eta|+|\xi|)^{m}}\Vert {e^{-t a(\xi) a(\xi)^*}}e_\eta \Vert_{L^2(\mathbb{T}^m)}  <\infty.$$
First, we observe that
\begin{align*}
   \Vert {e^{-t a(\xi)^* a(\xi)}}e_\eta \Vert_{L^2(\mathbb{T}^m)}= \Vert e_{-\eta}{e^{-t a(\xi)^* a(\xi)}}e_\eta \Vert_{L^2(\mathbb{T}^m)}= \Vert \sigma_{t}(\xi,\eta) \Vert_{L^2(\mathbb{T}^m)}=|\sigma_{t}(\xi,\eta)|,
\end{align*}where $\sigma_{t}(\xi,\cdot)$ is the symbol of the operator $e^{-t a(\xi)^* a(\xi)}.$ The symbolic calculus in this case allows us to conclude that $\sigma_{t}(\xi,\eta)=e^{-t \sigma(\xi,\eta)^* \sigma(\xi,\eta)}=e^{-t |a(\xi,\eta)|^2}.$ The estimate \eqref{lowebounellip} implies that
\begin{eqnarray}
 \Vert {e^{-t a(\xi)^* a(\xi)}}e_\eta \Vert_{L^2(\mathbb{T}^m)}\leq e^{-tC_\sigma\langle \xi,\eta\rangle^{\varkappa}}.
\end{eqnarray} Consequently,
\begin{eqnarray}
 \sup_{0<t<\infty}e^{tC_\sigma\langle \xi,\eta\rangle^{\varkappa}}\Vert {e^{-t a(\xi)^* a(\xi)}}e_\eta \Vert_{L^2(\mathbb{T}^m)}<\infty.
\end{eqnarray} A similar analysis shows that 
\begin{equation}
 \sup_{0<t<\infty}e^{tC_\sigma\langle \xi,\eta\rangle^{\varkappa}}\Vert {e^{-t a(\xi) a(\xi)^*}}e_\eta \Vert_{L^2(\mathbb{T}^m)}=\sup_{0<t<\infty}e^{tC_\sigma\langle \xi,\eta\rangle^{\varkappa}}\Vert {e^{-t a(\xi)^* a(\xi)}}e_\eta \Vert_{L^2(\mathbb{T}^m)}<\infty,\,\,\,\,\,
\end{equation} Nevertheless, it is  known fact that the index of periodic elliptic Fourier multipliers is null.
\end{example}
\section*{Acknowledgment}
We want to thank the anonymous referee for his/her valuable comments which improves the presentation of the paper. Vishvesh Kumar  wishes to thank Council of Scientific and Industrial Research, India, for its research grant. Currently, 	the authors are supported by FWO Odysseus 1 grant G.0H94.18N: Analysis and Partial Differential Equations of Prof. Michael Ruzhansky.

\bibliographystyle{amsplain}

\end{document}